\documentclass{amsart}
\usepackage{amsmath,amssymb,amsthm,amsfonts,amscd,mathrsfs,mathtools}
\usepackage[utf8]{inputenc}
\usepackage{graphicx}
\usepackage{enumitem}
\usepackage[hyphens]{url}
\usepackage{caption}
\usepackage{hyperref}
\usepackage{tikz-cd}

\theoremstyle{plain}
    \newtheorem{thm}{Theorem}[section]
    \newtheorem{lem}[thm]   {Lemma}

\theoremstyle{definition}
    \newtheorem{defn}[thm]  {Definition}
    
    \newtheorem{ex}{Example}

\newcommand{\C}{\mathbb{C}}
\newcommand{\cC}{\mathfrak{C}}
\newcommand{\fc}{\mathfrak{c}}
\newcommand{\D}{\mathcal{D}}
\newcommand{\F}{\mathbb{F}}
\newcommand{\cF}{\mathcal{F}}
\renewcommand{\H}{\mathcal{H}}
\newcommand{\M}{\mathcal{M}}
\newcommand{\N}{\mathcal{N}}
\newcommand{\p}{p}
\newcommand{\cP}{\mathcal{P}}
\newcommand{\Q}{\mathbb{Q}}

\renewcommand{\S}{\mathcal{S}}
\newcommand{\fS}{\mathfrak{S}}
\newcommand{\T}{\mathcal{T}}
\newcommand{\U}{\mathcal{U}}
\renewcommand{\u}{\mathfrak{u}}
\newcommand{\V}{\mathcal{V}}
\renewcommand{\v}{\mathfrak{v}}
\newcommand{\Z}{\mathbb{Z}}

\newcommand{\sgm}{\Sigma_{g,1}^m}
\newcommand{\sg}{\Sigma_{g,1}}

\renewcommand{\gg}{\Gamma_{g,1}}
\newcommand{\ggm}{\gg^m}

\newcommand{\bms}{\beta_{m}(\sg)}

\newcommand{\cms}{C_m(\S)}

\newcommand{\cmsD}{C_{m}^{\D}(\S)}
\newcommand{\fmstwoD}{C_{1,m}^{(2),\D}(\S)}
\newcommand{\cmstwoD}{C_{m+1}^{(2),\D}(\S)}
\newcommand{\Dmone}{\triangle_{m+1}}
\newcommand{\Donem}{\triangle_{1,m}}

\renewcommand{\j}{j}
\newcommand{\pr}{\mathfrak{p}}

\newcommand{\tup}{\mathfrak{h}}

\newcommand{\Ch}{\mathcal{C}}
\newcommand{\tCh}{\tilde\Ch}
\newcommand{\tH}{\tilde{H}}

\newcommand{\ZcC}[1]{\Z_2\left<\cC^{#1}\right>}
\newcommand{\tpsi}{\tilde\psi}

\newcommand{\pa}[1]{\left(#1\right)}

\newcommand{\set}[1]{\left\{#1\right\}}

\newcommand{\Sone}{\mathbb{S}^1}
\newcommand{\mrS}{\mathring{\S}}

\newcommand{\ux}{\underline{x}}
\newcommand{\uu}{\underline{u}}
\newcommand{\uv}{\underline{v}}
\newcommand{\ualpha}{\underline{\alpha}}

\newcommand{\tildeu}{\tilde u}
\newcommand{\tildev}{\tilde v}

\newcommand{\SP}{S\!P}

\renewcommand{\phi}{\varphi}
\renewcommand{\epsilon}{\varepsilon}

\newcommand{\hor}{hor}

\DeclareMathOperator{\Diff}{Diff}

\DeclareMathOperator{\Sym}{Sym}

\DeclareMathOperator{\Hom}{Hom}
\def\colim{\mathop{\mathrm{colim}}\nolimits}
    
\begin{document}

\title{Splitting of the homology of the punctured mapping class group}

\author{Andrea Bianchi}

\address{Mathematics Institute,
University of Bonn,
Endenicher Allee 60, Bonn,
Germany
}

\email{bianchi@math.uni-bonn.de}

\date{\today}

\keywords{Mapping class group, braid group, symplectic coefficients.}
\subjclass[2010]{20F36, 
55N25, 
55R20, 
55R35, 
55R40, 
55R80, 
55T10. 
}

\begin{abstract}
Let $\Gamma_{g,1}^m$ be the mapping class group of the orientable surface $\Sigma_{g,1}^m$ of genus $g$ with one parametrised
boundary curve and $m$ permutable punctures; when $m=0$ we omit it from the notation.
Let $\beta_{m}(\Sigma_{g,1})$ be the braid group on $m$ strands
of the surface $\Sigma_{g,1}$.

\noindent We prove that $H_*(\Gamma_{g,1}^m;\mathbb{Z}_2)\cong H_*(\Gamma_{g,1};H_*(\beta_{m}(\Sigma_{g,1});\mathbb{Z}_2))$. The main ingredient
is the computation of $H_*(\beta_{m}(\Sigma_{g,1});\mathbb{Z}_2)$ as a symplectic representation of $\Gamma_{g,1}$.
\end{abstract}

\maketitle

\section{Introduction}
Let $\sg$ be a smooth orientable surface of genus $g$ with one boundary curve $\partial\sg$, and let $\sgm$ be $\sg$ with
a choice of $m$ distinct points in the interior, called \emph{punctures}.

Let $\gg$ be the mapping class group of $\sg$, i.e. the group of isotopy classes of diffeomorphisms of $\sg$:
diffeomorphisms are required to fix $\partial\sg$ pointwise. Similarly let $\ggm$ be the mapping class group of $\sgm$, i.e.
the group of isotopy classes of diffeomorphisms of $\sgm$ that fix $\partial\sgm$ pointwise and \emph{permute} the $m$ punctures.

Forgetting the punctures gives a surjective map $\ggm\to\gg$ with kernel $\bms$,
the $m$-th \emph{braid group} of the surface $\sg$. We obtain the Birman exact sequence (see \cite{Birman:mcgbr})
\begin{equation}
\label{eq:Birman}
1\to\bms\to\ggm\to\gg\to 1.
\end{equation}

The associated Leray-Serre spectral sequence $E(m)$ in $\Z_2$-homology has a second page $E(m)^2_{k,q}=H_k(\gg;H_q(\bms;\Z_2))$,
and converges to $H_{k+q}(\ggm;\Z_2)$.

The main result of this article is that this spectral sequence collapses in its second page.
\begin{thm}
\label{thm:main}
For all $l\geq 0$ there is an isomorphism of vector spaces
\begin{equation}
\label{eq:main}
H_l\pa{\ggm;\Z_2}\cong \bigoplus_{k+q=l} H_k\pa{\gg;H_q\pa{\bms;\Z_2}}.
\end{equation}
\end{thm}
Thus the computation of $H_*\pa{\ggm;\Z_2}$ reduces to the computation of the homology of $\gg$ with
twisted coefficients in the representation $H_*\pa{\bms;\Z_2}$. We will see that this $\gg$-representation
splits as a direct sum of symmetric powers of $H_1(\sg)$ with the symplectic action: this is done in
Theorem \ref{thm:Hbms*as*ggrep}, which together with Theorem \ref{thm:main} is the main result of the article.

The strategy of the proof does not generalize to fields of characteristic different from 2 or to the \emph{pure}
mapping class group, in which we consider only diffeomorphisms of $\sgm$ that fix all punctures. In Section \ref{sec:rational}
we describe in detail a counterexample with coefficients in $\Q$, which can be generalized both to coefficients
in a field $\F_p$ of odd characteristic and to the pure mapping class group.

I would like to thank my PhD advisor Carl-Friedrich B\"odigheimer for his precious suggestions and his continuous encouragement
during the preparation of this work.

\section{Preliminaries}
\label{sec:Preliminaries}
In the whole article $\Z_2$-coefficients for homology and cohomology will be understood,
unless explicitly stated otherwise.

In this section we recollect some classical definitions and results about braid groups and mapping class groups.

\begin{defn}
\label{defn:cms}
 The $m$-th \emph{ordered configuration space} of a surface $\sg$ is the space
\[
 F_m(\sg)=\set{(P_1,\dots,P_m) \in \pa{\mathring{\Sigma}_{g,1}}^{\times m}  \,|\,  P_i\neq P_j  \;\forall i\neq j}.
\]

 Note that we require the points of the configuration to lie in the interior of $\sg$; the space
 $F_m(\sg)$ is a smooth, orientable $2m$-dimensional manifold.
 
 The symmetric group $\mathfrak{S}_m$ acts freely on $F_m(\sg)$ by permuting the labels $1,\dots,m$ of a configuration;
 the orbit space
 \[
 F_m(\sg)/\mathfrak{S}_m
 \]
 is called the \emph{$m$-th unordered configuration space}
 of $\sg$ and is denoted by $C_m(\sg)$; this space is also a $2m$-dimensional orientable manifold
 (see Figure \ref{fig:unordered}).
 
\end{defn}

\begin{figure}[ht]\centering
\includegraphics[scale=0.8]{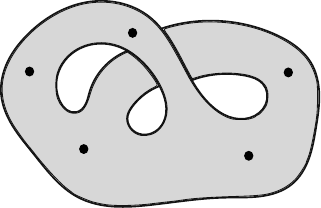}
\caption{A configuration in the space $C_5(\Sigma_{1,1})$}
\label{fig:unordered}
\end{figure}

A classical result by Fadell and Neuwirth (\cite{FadellNeuwirth}) ensures
that $C_m(\sg)$ is aspherical; the fundamental group $\pi_1(C_m(\sg))$ is
called the \emph{braid group on $m$ strands of $\sg$} and is denoted by $\bms$.

\begin{defn}
 \label{def:mcg}
 Let $\Diff(\sg;\partial\sg)$ be the group of diffeomorphisms of $\sg$ that fix $\partial\sg$ pointwise,
 endowed with the Whitney $C^{\infty}$-topology. We denote by $\gg=\pi_0(\Diff(\sg;\partial\sg))$ its group of connected
 components, which is called the \emph{mapping class group} of $\sg$.
 
 Similarly let $\Diff(\sgm;\partial\sgm)$ be the group of diffeomorphisms of $\sgm$ that fix $\partial(\sg)$
 pointwise and restrict to a permutation of the $m$ punctures: the \emph{mapping class group} of $\sgm$ is
 $\ggm=\pi_0(\Diff(\sgm;\partial\sgm))$.
 \end{defn}
 A classical result by Earle and Schatz \cite{EarleSchatz} ensures that the connected components
 of $\Diff(\sg;\partial\sg)$ are contractible: note that for $g=0$ and $g=1$ this result holds
 because we consider surfaces with non-empty boundary, and the bounday must be fixed pointwise
 by our diffeomorphisms. In particular $B\Diff(\sg;\partial\sg)\simeq B\gg$ is
 a classifying space for
 the group $\gg$, i.e. an Eilenberg-MacLane space of type $K(\gg,1)$.
 
 We call $\cF_{g,1}\to B\Diff(\sg;\partial\sg)$ the universal
 $\sg$-bundle
 \[
  \sg\to\cF_{g,1}=\sg\times_{\Diff(\sg;\partial\sg)}E\Diff(\sg;\partial\sg)\to B\Diff(\sg;\partial\sg).
 \]
Applying the construction of the $m$-th unordered configuration space fiberwise we obtain a bundle
$C_m(\cF_{g,1})\to B\Diff(\sg;\partial\sg)$ with fiber $C_m(\sg)$.

The space $C_m(\cF_{g,1})$ is a classifying space for the
group $\ggm$ and the Birman exact sequence \eqref{eq:Birman} is obtained by taking fundamental
groups of the aspherical spaces
\begin{equation}\label{eq:Birmanbundle}
C_m(\sg)\to C_m(\cF_{g,1})\to B\Diff(\sg;\partial\sg).
\end{equation}

In the whole article the genus $g\geq 0$ of the surfaces that we consider is supposed to be fixed,
and we will abbreviate $\S=\sg$.
We denote by $\D$ the open disc $\mathring{\Sigma}_{0,1}$.

It will be useful, for many constructions, to choose an embedding
$\D\hookrightarrow\mrS$ \emph{near} $\partial\S$ and to replace
$\Diff(\S;\partial\S)$ with its subgroup $\Diff(\S;\partial\S\cup\D)$ of diffeomorphisms
of $\S$ that fix pointwise both $\partial\S$ and $\D$. Saying that $\D$ is embedded
\emph{near $\partial\S$} means that there is a compact subsurface $\S'\subset\S$
such that $\S=\S'\natural\bar\D$ is the union along a segment of $\S'$ and the closure of $\D$ in $\S$.

From now on we suppose such an embedding to be fixed and we consider $\D$ as a subspace
of $\mrS$ (see Figure \ref{fig:SS'D}). In Section \ref{sec:HBraidSurf}, Definition \ref{defn:TS},
we will introduce a convenient model $\T(\S)$ for the space $\mrS$, and in Section \ref{sec:Actiongg},
Definition \ref{defn:DinTS}
we will specify an embedding $\D\hookrightarrow\mrS$ using the model $\T(\S)$.

\begin{figure}[ht]\centering
 \includegraphics[scale=0.8]{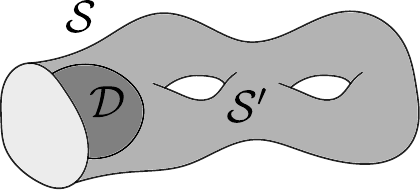}
 \caption{A splitting of $\S$ as $\S'\natural\bar\D$.}
\label{fig:SS'D}
\end{figure}

The construction in Definition \ref{defn:cms} can be specialised to the
surfaces $\D$ and $\S'$, yielding spaces $C_m(\D)$ and $C_m(\S')$ respectively.
We take configurations of points in the interior of the surfaces $\D$ and $\S'$ respectively
(since $\D$ is already an open surface this remark applies in particular to $\S'$).

The inclusion $\Diff(\S;\partial\S\cup\D)\subset\Diff(\S;\partial\S)$ is a homotopy
equivalence of topological groups, hence also the induced map
\[
B\Diff(\S;\partial\S\cup\D)\to B\Diff(\S;\partial\S)
\]
is a homotopy equivalence. We will replace the previous construction with the following,
homotopy equivalent ones.
\begin{defn}
\label{defn:universalSbundle}
We call $\cF_{\S,\D}\to B\Diff(\S;\partial\S\cup\D)$
the universal $\S$-bundle
 \[
  \S\to\cF_{\S,\D}\colon=\S\times_{\Diff(\S;\partial\S\cup\D)}E\Diff(\S;\partial\S\cup\D)\to B\Diff(\S;\partial\S\cup\D).
 \]
Note that $\Diff(\S;\partial\S\cup\D)$ acts on $\S'$ by restriction of diffeomorphisms, and acts trivially
on $\D$; therefore $\cF_{\S,\D}$ contains subspaces
\[
 \cF_{\S'}= \S'\times_{\Diff(\S;\partial\S\cup\D)}E\Diff(\S;\partial\S\cup\D)
\]
and $\D\times B\Diff(\S;\partial\S\cup\D)$; these subspaces fiber over $B\Diff(\S;\partial\S\cup\D)$
with fibers $\S'$ and $\D$ respectively.

Applying the construction of the $m$-th unordered configuration space fiberwise, we obtain spaces
$C_m(\cF_{\S,\D})$, $C_m(\cF_{\S'})$ and $C_m(\D)\times B\Diff(\S;\partial\S\cup\D)$,
all fibering over $B\Diff(\S;\partial\S\cup\D)$, respectively with fiber $\cms$,
$C_m(\S')$ and $C_m(\D)$.
\end{defn}
The fiber bundle \eqref{eq:Birmanbundle} corresponding to the Birman exact sequence \eqref{eq:Birman}
can now be replaced with the following one
\begin{equation}\label{eq:BirmanbundleD}
\cms\to C_m(\cF_{\S,\D})\to B\Diff(\S;\partial\S\cup\D).
\end{equation}

\begin{defn}
\label{defn:muF}
For all $0\leq p\leq m$ there is a map
\[
 \mu\colon C_p(\D)\times C_{m-p}(\S')\to \cms
\]
which takes the union of configurations (see Figure \ref{fig:defmu}):
 \[
  \mu\pa{\set{P_1,\dots,P_p};\set{P'_1,\dots,P'_{m-p}}}=\set{P_1,\dots,P_p,P'_1,\dots,P'_{m-p}}\in\cms.
 \]
\begin{figure}[ht]\centering
 \includegraphics[scale=0.7]{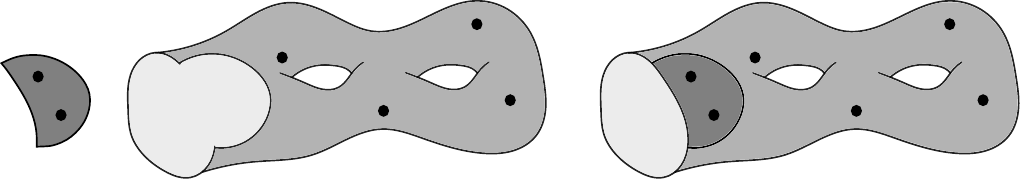}
 \caption{The product of two configurations in $C_2(\D)$ and $C_4(\S')$.}
\label{fig:defmu}
\end{figure}
 
We can apply this construction at the same time to each couple of fibers of the bundles $C_m(\cF_{\S,\D})$ and $C_{m-p}(\cF_{\S'})$
lying over the same point of $ B\Diff(\S;\partial\S\cup\D)$.

We obtain a map
\[
 \mu^{\cF}\colon C_p(\D)\times C_{m-p}(\cF_{\S'})\to C_m(\cF_{\S,\D}).
\]
If we see the domain of $\mu^{\cF}$ as a fibered product of bundles over $B\Diff(\S;\partial\S\cup\D)$
\[
 \pa{C_m(\D)\times B\Diff(\S;\partial\S\cup\D)}\times_{B\Diff(\S;\partial\S\cup\D)}C_{m-p}(\cF_{\S'}),
\]
then the map $\mu^{\cF}$ is also a map of bundles over $B\Diff(\S;\partial\S\cup\D)$, and the corresponding
map on fibers is precisely $\mu$.
\end{defn}

We recall now the structure of $H_*(C_m(\D))$.
The cohomology of $C_m(\D)$ with coefficients in $\Z_2$ was first computed by Fuchs in
\cite{Fuchs:CohomBraidModtwo}: in section
\ref{sec:HBraidSurf} we will generalise  Fuchs' argument to surfaces with boundary of positive genus.

In the appendix of \cite[Chap.III]{CLM}, Cohen considers the space $\coprod_{m\geq 0}C_m(\D)$ as an
algebra over the operad of little $2$-cubes, and describes its
$\Z_2$-homology as follows:
\begin{equation}
 \label{eq:Cohen}
H_*\pa{\coprod_{m\geq 0}C_m(\D)}\cong \Z_2\left[Q^j\epsilon \, |\, j\geq 0\right].
\end{equation}

Here $\epsilon\in H_0(C_1(\D))$ is the fundamental class, and for all
$k,m\geq 0$ we denote by $Q\colon H_k(C_m(\D))\to H_{2k+1}(C_{2m}(\D))$
the first Dyer-Lashof operation. In particular $Q^j\epsilon$ is the generator of
$H_{2^j-1}(C_{2^j}(\D))\simeq\Z_2$. See Figure \ref{fig:monomial}.

The isomorphism in equation \eqref{eq:Cohen} is
an isomorphism of bigraded rings. The left-hand side is a ring with the Pontryagin product,
and the right-hand side is a polynomial ring in infinitely many variables $\epsilon,Q\epsilon,Q^2\epsilon,\dots$.
The bigrading is given by the homological degree
$*$, that we call the \emph{degree},
and by the index $m$ of the connected component
on which the homology class is supported (informally, the number of points
involved in the construction of the homology class), that we call the \emph{weight}.

 \begin{figure}[ht]\centering
 \includegraphics[scale=0.7]{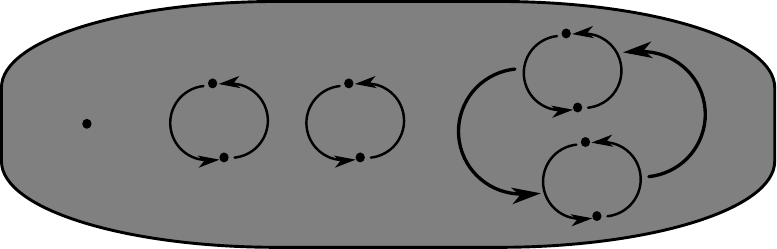}
 \caption{The class in $H_5(C_9(\D))$ corresponding to the monomial $\epsilon\cdot(Q\epsilon)^2\cdot Q^2\epsilon$.}
\label{fig:monomial}
\end{figure}

In this article we will only need the isomorphism in equation \eqref{eq:Cohen} to hold as
an isomorphism of bigraded $\Z_2$-vector spaces.
In particular for all choices of natural numbers $(\alpha_j)_{j\geq 0}$ with all but finitely
many $\alpha_j=0$, we will consider the monomial $\prod_j\pa{Q^j\epsilon}^{\alpha_j}$, corresponding to
a homology class in some bigrading $(*,m)$: we will only use the fact
that the set of these monomials is a bigraded basis for the
left-hand side of equation \eqref{eq:Cohen}.

We end this section recalling the classical notions of symmetric product and one-point compactification.

\begin{defn}
 \label{defn:SP}
 The $m$-fold symmetric product of a topological space $\mathcal{M}$, denoted by $\SP^m(\mathcal{M})$, is
 the quotient of $\mathcal{M}^m$ by the (non-free) action of the symmetric group $\mathfrak{S}_m$
 which permutes the coordinates.
\end{defn}
We regard a point of $\SP^m(\mathcal{M})$ as a configuration of $m$ points in $\mathcal{M}$ with multiplicities.

In the case of an open surface $\mrS$, it is known that $\SP^m(\mrS)$ is a non-compact manifold of dimension $2m$;
it contains $\cms$ as an open submanifold.

\begin{defn}
 \label{defn:opc}
For a space $\mathcal{M}$ we denote by $\mathcal{M}^{\infty}$ its one-point
compactification; the basepoint is the point at infinity, that we denote by $\infty$.
\end{defn}

We will consider in particular the one-point-compactification $\cms^{\infty}$ of $\cms$.

\section{Homology of configuration spaces of surfaces}
\label{sec:HBraidSurf}
We want to study $H_*(\bms)=H_*(\cms)$,
which appears as the homology of the fiber
in the Leray-Serre spectral sequence associated to
\eqref{eq:Birman}, \eqref{eq:Birmanbundle} or \eqref{eq:BirmanbundleD}.
In particular we are interested in $H_*(C_m(\S))$
as a $\Z_2$-representation of the group $\gg$.

In \cite{LM} L\"offler and Milgram implicitly proved that $H_*(\bms)=H_*(\cms)$ is a $\Z_2$-symplectic
representation of the mapping class group. By \emph{$\Z_2$-symplectic} we mean the following:
\begin{defn}
 \label{defn:symplrep}
 Let $\H=H_1(\S)\simeq\Z_2^{2g}$.
 The natural action of $\gg$ on $\H$ induces a surjective map
 $\gg\to Sp_{2g}(\Z_2)$. A representation of $\gg$ over $\Z_2$ is called \emph{$\Z_2$-symplectic}
 if it is a pull-back of a representation of $Sp_{2g}(\Z_2)$ along this map.
\end{defn}

In \cite{BCT} B\"odigheimer, Cohen and Taylor computed $H_*(\cms)$ \emph{as a graded $\Z_2$-vector space}.
Their method provides all Betti numbers, but the action of $\gg$ cannot be easily deduced:
their descripition of $H_*(\cms)$ depends on a handle decomposition of $\S$, which is not preserved
by diffeomorphisms of $\S$, not even up to isotopy.

In this section and in the next one we will prove the following theorem; to the best of the author's knowledge
it does not appear in the literature.
\begin{thm}
 \label{thm:Hbms*as*ggrep}
 There is an isomorphism of bigraded $\Z_2$-representations of $\gg$
 \[
  \bigoplus_{m\geq 0} H_*\pa{\cms}\cong \Z_2\left[Q^j\epsilon\,|\, j\geq 0\right]\otimes\Sym_{\bullet}(\H).
 \]
 Here we mean the following:
 \begin{itemize}
  \item[(i)] on the left-hand side the bigrading is given by homological degree $*$ and by the direct summand,
  indexed by $m$,
  on which the homology class is supported, i.e. by the number $m$ of points
  involved in constructing the homology class; we call $*$ the \emph{degree} and $m$ the \emph{weight},
  and write $(*,m)$ for the bigrading;
  \item[(ii)] for $j\geq 0$, $Q^j\epsilon$ is the image in $H_{2^j-1}(C_{2^j}(\S))$ of a generator
  of the group $H_{2^j-1}(C_{2^j}(\D))\simeq \Z_2$ 
  under the natural map induced by the embedding $\D\hookrightarrow\mrS$,
  and $\Z_2\left[Q^j\epsilon\,|\, j\geq 0\right]$ is the polynomial ring on
  infinitely many variables $\epsilon,Q\epsilon,Q^2\epsilon,\dots$;
  \item[(iii)] $\H=H_1(\sg)$ is identified with $H_1(C_1(\S))$ in a natural way, and lives in degree $1$ and weight $1$;
  $\Sym_{\bullet}(\H)$ denotes the symmetric algebra on $\H$;
  \item[(iv)] degrees and weights are extended on the right-hand side by the usual multiplicativity rule;
  \item[(v)] the action of $\gg$ on the right is the tensor product of the trivial action
  on the factor $\Z_2[Q^j\epsilon\,|\, j\geq 0]$, and of the action on
  $\Sym_{\bullet}(\H)$ which is induced by the $\Z_2$-symplectic action on $\H$.
  \end{itemize}
\end{thm}
Note that for any bi-homogeneus element in the right-hand side, the weight is greater or equal than
the degree: indeed factors of the form $Q^j\epsilon$ have weight strictly higher than their degree,
whereas factors belonging to $\H$ or to its symmetric powers have equal weight and degree.
  
Note that in the case $g=0$ the group $\Gamma_{0,1}$ is trivial and the previous theorem
reduces to equation \eqref{eq:Cohen}.

We point out that in \cite{BoT} B\"{o}digheimer and Tillmann have essentially proved that for a field
$\mathbb{F}$ of any characteristic the $\mathbb{F}$-vector space
\[
  \bigoplus_{m\geq 0} H_*\pa{\cms;\mathbb{F}}
\]
is isomorphic, as a bigraded $\gg$-representation over $\mathbb{F}$, to the tensor product
of the ring $\mathbb{F}[\epsilon]$, with trivial action, and some other bigraded representation:
here $\epsilon$ denotes, in analogy with the notation of Theorem \ref{thm:Hbms*as*ggrep}, the standard generator of $H_0(C_1(\D))$.
In Section \ref{sec:comparison} we will compare in detail Theorems \ref{thm:Hbms*as*ggrep}
and \ref{thm:main} with the results in \cite{BoT}.

In this section we will prove that there is an isomorphism of \emph{bigraded $\Z_2$-vector spaces}
as in Theorem \ref{thm:Hbms*as*ggrep}; in the next section we will deal with the action of
$\gg$.

Since we work with coefficients in a field, it is equivalent to compute homology or cohomology,
and in this section we will prefer to compute $H^*(\cms)$ for all bigradings $(*,m)$.

We will mimic the method used by Fuchs \cite{Fuchs:CohomBraidModtwo} to compute the $\Z_2$-cohomology
of $C_m(\D)$.
As already mentioned in Section \ref{sec:Preliminaries}, our computation recovers a known result, but it has the advantage of
being quite elementary and of providing a part of
the geometric insight that we will need in the next section.

In the whole section we assume $m\geq 0$ to be fixed.
Since the space $\cms$ is homeomorphic to the interior of a compact
$2m$-manifold with boundary, by Poincaré-Lefschetz
duality we have
\[
 H^*(\cms)\simeq \tilde H_{2m-*}(\cms^{\infty}),
\]
where in the right hand side we consider reduced homology of the one-point compactification (see Definition \ref{defn:opc}).

We introduce a space $\T(\S)$ which is homeomorphic to $\mrS$, the interior of $\S$. The construction
corresponds to a handle decomposition of $\S$ with one $0$-handle and $2g$ $1$-handles.

\begin{defn}
\label{defn:TS}
If $g=0$, hence $\S=\Sigma_{0,1}$ is the disc, we set $\T(\S)=]0,1[^2$, the interior of the unit square. Assume
now $g\geq 1$, and see Figure \ref{fig:defTS} to visualize the following construction.

Dissect the interval $[0,1]$ into $2g$ equal subintervals through the points $\cP_i=\frac{i}{2g}$ for $0\leq i\leq 2g$
(for $i=0,2g$ we get the two endpoints of $[0,1]$).

Consider on the vertical sides of $[0,1]^2$ the intervals
$I_i^l=\set{0}\times [\cP_i,\cP_{i+1}]$ and $I_i^r=\set{1}\times [\cP_i,\cP_{i+1}]$ for $1\leq i\leq 2g$:
all these intervals are
canonically diffeomorphic
to $[0,1]$ by projecting on the second coordinate, rescaling linearly by a factor $2g$
and translating.

We define a bijection
between the two sets of left and right intervals:
for $1\leq i\leq g$, the interval $I^l_{2i-1}$ corresponds to $I^r_{2i}$,
and the interval $I^r_{2i-1}$ corresponds $I^l_{2i}$.

The space $Q(\S)$ is the quotient of the square $[0,1]^2$ obtained by identifying each couple
of corresponding intervals in the canonical way. For $1\leq i\leq g$ we call $\overline{\U}_i$ the image
of $I^l_{2i-1}$ in the quotient $Q(\S)$, and we call $\overline{\V}_i$ the image of $I^l_{2g}$ in $Q(\S)$.

Note that the image in $Q(\S)$ of the set $\set{0,1}\times\set{\cP_0,\dots,\cP_{2g}}$ consists of two
points $\cP_{odd}$ and $\cP_{even}$: for $\epsilon\in\set{0,1}$ and $0\leq i\leq 2g$ the point
$(\epsilon,\cP_i)$ is mapped to $\cP_{even}$ if $\epsilon+i$ is even, and is mapped to $\cP_{odd}$
otherwise.

The spaces $\overline{\U}_i\subset Q(\S)$ and $\overline{\V}_i\subset Q(\S)$ are intervals with endpoints $\cP_{even}$ and $\cP_{odd}$;
the interiors of these intervals are disjoint. Each $\overline{\U}_i$ and $\overline{\V}_i$ is the homeomorphic image of
some left interval $I^l_j$, and inherits from the latter a parametrisation by $[0,1]$.

We call $\U_i$ and $\V_i$ the interiors of the intervals $\overline{\U}_i$ and $\overline{\V}_i$ respectively.

The space $Q(\S)$ is homeomorphic to the compact surface $\S$; we call $\T(\S)$ the interior of
$Q(\S)$.

From now on we will identify $\S$ with $Q(\S)$ and $\mrS$ with $\T(\S)$;
consequently we will identify $\cms$ with the
space of configurations of $m$ points in $\T(\S)$.
\end{defn}

\begin{figure}[ht]\centering
 \includegraphics[scale=0.7]{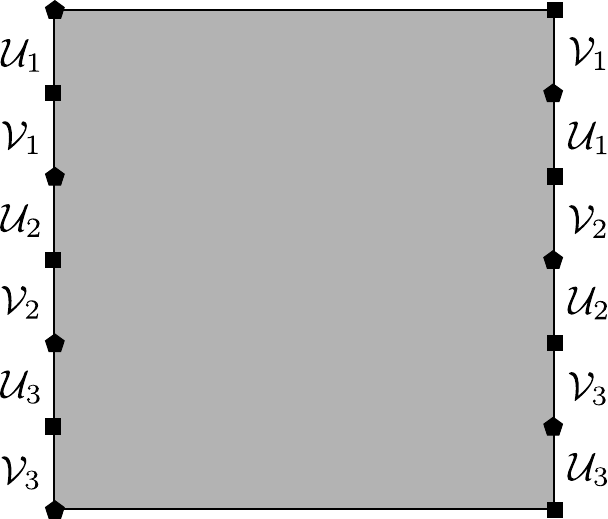}
 \caption{The space $Q(\S)$ for $\S=\Sigma_{3,1}$. The black pentagons represent the point $\cP_{even}$, whereas the black
 squares represent $\cP_{odd}$.}
\label{fig:defTS}
\end{figure}

Our next aim is to define a structure of $CW$-complex on the space $\cms^{\infty}$:
the only $0$-cell will be the point $\infty$, whereas the other cells will be introduced
in the following definition.
\begin{defn}
\label{defn:ehopen}
A \emph{tuple} $\tup$ is a choice of the following set of data:
 \begin{itemize}
  \item a natural number $0\leq l\leq m$;
  \item a vector $\ux=(x_1,\dots, x_l)$ of integers $\geq 1$;
  \item vectors $\uu=(u_1,\dots,u_g)$ and $\uv=(v_1,\dots,v_g)$ of integers $\geq 0$,
 \end{itemize}
satisfying the following equality
\[
 m=\sum_{i=1}^lx_i+\sum_{i=1}^g(u_i+v_i).
\]
In symbols we write $\tup=(l,\ux,\uu,\uv)$. We omit from the notation $\ux$ if it vanishes: this happens precisely when $l=0$.
The \emph{dimension} of $\tup$ is defined as $m+l$.

For a tuple $\tup$ let $e^{\tup}$ be the subspace
of $\cms$ of configurations of $m$ points in $\mrS$ such that the following conditions hold
(see picture \ref{fig:defetup}):
\begin{itemize}
 \item for all $1\leq i\leq g$, exactly $u_i$ points lie on $\U_i$
 and exactly $v_i$ points lie on $\V_i$;
 \item there are exactly $l$ vertical lines in the open square $]0,1[^2\subset\mrS$ of the
 form $\set{s_i}\times]0,1[$ for some $0<s_1<\dots<s_l<1$, containing at least one
 point of the configuration. From left to right, these lines contain exactly $x_1,\dots,x_l$ points
 respectively.
\end{itemize}
\end{defn}

\begin{figure}[ht]\centering
 \includegraphics[scale=0.7]{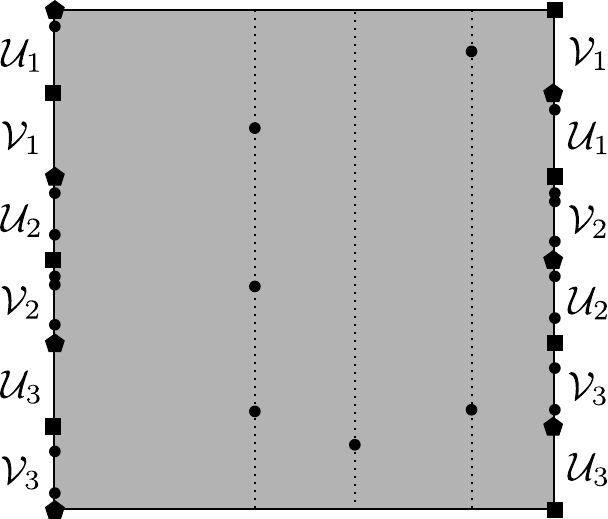}
 \caption{A configuration in the subspace $e^{\tup}\subset C_{14}(\S)$, for the tuple $\tup=(3,(3,1,2),(1,2,0),(0,3,2))$. Points
 lying on the $\U_i$'s and the $\V_i$'s are depicted twice.}
\label{fig:defetup}
\end{figure}

The space $e^{\tup}$ is homeomorphic to \emph{the interior} of the following multisimplex:

\[
 \Delta^{\tup}\colon =
 \Delta^l\times\prod_{i=1}^l\Delta^{x_i}\times\prod_{i=1}^g\pa{\Delta^{u_i}\times\Delta^{v_i}},
\]
where the simplex $\Delta^r$ is the subspace of $[0,1]^r$ of sequences $0\leq \tau_1\leq\dots\leq\tau_r\leq 1$
(the numbers $\tau_1,\dots,\tau_r$ are the \emph{local coordinates} of the simplex). The homeomorphism is given
as follows:
\begin{itemize}
 \item the local coordinates of the $\Delta^l$-factor correspond to the positions $s_1,\dots,s_l$ of the vertical
 lines in $(0,1)^2$ containing points of the configuration;
 \item the local coordinates of the $\Delta^{x_i}$-factor correspond to the positions of the $x_i$ points
 lying on the vertical line $\set{s_i}\times(0,1)$;
 \item the local coordinates of the $\Delta^{u_i}$-factor correspond to the positions of the $u_i$ points
 lying on $\U_i$, which is canonically identified with $(0,1)$; similarly for the $\Delta^{v_i}$-factor,
 with $v_i$ and $\V_i$ replacing $u_i$ and $\U_i$.
\end{itemize}
Note that the dimension of $\Delta^{\tup}$ agrees with the dimension of $\tup$ from Definition \ref{defn:ehopen}.
The embedding $\mathring{\Delta}^{\tup}\cong e^{\tup}\hookrightarrow \cms^{\infty}$ extends to a continuous map
$\Phi^{\tup}\colon\Delta^{\tup}\to\cms^{\infty}$, so that the image of $\partial\Delta^{\tup}$ is contained in the union of
all subspaces $e^{\tup'}$ for tuples $\tup'$ of \emph{lower} dimension than $\tup$, together with the $0$-cell $\infty$.

The construction of the map $\Phi^{\tup}$ is as follows:
\begin{enumerate}
\item we identify the one-point compactification
$(\mrS)^{\infty}$ of $\mrS=\T(\S)$ as the quotient of $Q(\S)$ collapsing the boundary to a point $\infty$;
\item we consider the $m$-fold symmetric product $\SP^m\pa{(\mrS)^{\infty}}$ (see Definition \ref{defn:SP}): it contains 
as an \emph{open} subspace $\cms$, so we can identify $\cms^{\infty}$ as the quotient of $\SP^m\pa{(\mrS)^{\infty}}$ collapsing
the subspace $\SP^m\pa{(\mrS)^{\infty}}\setminus\cms$ to $\infty$;
\item the homeomorphism $\mathring{\Delta}^{\tup}\to e^{\tup}\subset\cms$ extends
now to a map $\Delta^{\tup}\to\SP^m\pa{[0,1]^2}$, that we can then further project to $\SP^m(Q(\S))$, then to
$ \SP^m\pa{(\mrS)^{\infty}}$ and then to $\cms^{\infty}$: the composition
is the map $\Phi^{\tup}$.
\end{enumerate}
In the last step we used the open inclusion $\cms\subset\SP^m\pa{\mrS}\subset\SP^m\pa{(\mrS)^{\infty}}$, which
induces a map $ \SP^m\pa{(\mrS)^{\infty}}\to\cms^{\infty}$ by quotienting to $\infty$ the closed subspace
$ \SP^m\pa{(\mrS)^{\infty}}\setminus \cms^{\infty}$.

We conclude that the collection of the $e^{\tup}$'s, together with the $0$-cell $\infty$,
gives a cell decomposition of $\cms^{\infty}$, with characteristic maps of cells $\Phi^{\tup}$.

We compute the \emph{reduced} cellular chain complex of $\cms^{\infty}$
with coefficients in $\Z_2$. It is a chain complex in $\Z_2$-vector spaces with basis given
by tuples $\tup$, which correspond to the cells $e^{\tup}$.
\begin{lem}
\label{lem:doperatoropenmodtwo}
Let $\tup=(l,\ux,\uu, \uv)$ and
$\tup'=(l-1,\ux',\uu',\uv')$
be tuples in consecutive dimensions $m+l$ and $m+l-1$. Denote by $[\tup\colon\tup']\in\Z_2$ the coefficient
of $\tup'$ in $\partial \tup$ in the reduced chain complex $\tilde\Ch_*(\cms^{\infty})$.
Then $[\tup:\tup']=0$ unless one, and exactly one, of the following situations occurs:
\begin{itemize}
 \item $l\geq 2$ and $\tup'$ is obtained from $\tup$ by decreasing $l$ by 1, setting $x'_i=x_i+x_{i+1}$
 for one value $1\leq i\leq l-1$, and shifting the values
 $x'_j=x_{j+1}$ for $i+1\leq j\leq l-1$. In this case we say that $\tup'$ is an \emph{inner boundary} of $\tup$, and
 we have
 \[
  [\partial \tup\colon\tup']=\binom{x_i+x_{i+1}}{x_i}\in\Z_2,
 \]
 where $\binom{x_i+x_{i+1}}{x_i}$ denotes the binomial coefficient.
 \item $l\geq 1$ and $\tup'$ is obtained from $\tup$ by decreasing $l$ by 1, choosing a splitting of $x_1$
 in integers $\tildeu_i,\tildev_i\geq 0$
 \[
  x_1=\sum_{i=1}^g \pa{\tildeu_i+\tildev_i},
 \]
 setting $u'_i=u_i+\tildeu_i$ and $v'_i=v_i+\tildev_i$ for all $1\leq i\leq g$ and shifting the values
 $x'_j=x_{j+1}$ for $1\leq j\leq l-1$. In this case we say that $\tup'$ is a \emph{left, outer boundary} of $\tup$, and we have
 \[
  [\partial \tup\colon\tup']=\prod_{i=1}^g\binom{u_i+\tildeu_i}{u_i}\binom{v_i+\tildev_i}{v_i}\in\Z_2.
 \]
 \item $l\geq 1$ and $\tup'$ is obtained from $\tup$ by decreasing $l$ by 1, choosing a splitting of $x_l$
 in integers $\tildeu_i,\tildev_i\geq 0$
 \[
  x_l=\sum_{i=1}^g \pa{\tildeu_i+\tildev_i},
 \]
 setting $u'_i=u_i+\tildeu_i$ and $v'_i=v_i+\tildev_i$ for all $1\leq i\leq g$ and keeping $x'_i=x_i$ for all $1\leq i\leq l-1$.
 In this case we say that $\tup'$ is a \emph{right, outer boundary} of $\tup$, and we have
 \[
  [\partial \tup\colon\tup']=\prod_{i=1}^g\binom{u_i+\tildeu_i}{u_i}\binom{v_i+\tildev_i}{v_i}\in\Z_2.
 \]
\end{itemize}

\end{lem}
It may indeed happen that $\tup'$ is both a left and a right outer boundary of $\tup$,
namely when all numbers $(x_i)_{1\leq i\leq l}$ are equal; then the two contributions cancel each other,
so that $[\tup\colon\tup']=0$ as stated.
\begin{proof}
 For $1\leq i\leq l$ denote by $\partial^{\hor}_i\Delta^{\tup}$ the face
 \[
 \partial_i\Delta^l\times\prod_{i=1}^l\Delta^{x_i}\times\prod_{i=1}^g\pa{\Delta^{u_i}\times\Delta^{v_i}}\subset\Delta^{\tup}.
 \]
 This is also referred as the $i$-th \emph{horizontal face}. All other faces of codimension 1 of
 the multisimplex $\Delta^{\tup}$ are called \emph{vertical}.
 
 We note that $\Phi^{\tup}$ restricts to a cellular map $\partial_i^{\hor}\Delta^{\tup}\to\cms^{\infty}$ on every
 horizontal face, where $\partial_i^{\hor}\Delta^{\tup}$
 is given the cell structure coming from the multisimplicial structure: every subface of $\partial_i^{\hor}\Delta^{\tup}$ of dimension
 $k\leq m+l-1$ is mapped to the $k$-skeleton of $\cms^{\infty}$. Therefore the map $\Phi^{\tup}\colon\partial_i^{\hor}\Delta^{\tup}\to\cms^{\infty}$
 has a well-defined local index over the cell $e^{\tup'}$, that
 we call $[\partial^{\hor}_i\tup\colon\tup']\in\Z_2$.

 The same holds for vertical faces, where the restriction of $\Phi^{\tup}$ is the constant map to $\infty$: in this case the local index
 over $e^{\tup'}$ is zero. The index $[\tup\colon\tup']$ of the map $\Phi^{\tup}\colon\partial\Delta^{\tup}\to\cms^{\infty}$ on
 $e^{\tup'}$ splits as a sum of local indices:
 \[
  [\tup\colon\tup']=\sum_{i=0}^l[\partial_i^{\hor}\tup\colon\tup']\in\Z_2.
 \]
 For $0\leq i\leq l$ the restriction $\Phi^{\tup}\colon\partial^{\hor}_i\Delta^{\tup}\to\cms^{\infty}$
 hits homeomorphically the open cell $e^{\tup'}$
 exactly as many times as specified in the statement
 of the lemma for the cases $1\leq i\leq l-1$,
 $i=0$ and $i=l$ respectively.
 
 The only possibility in which the same cell $e^{\tup'}$ is hit by different faces $\partial_i^{\hor}\Delta^{\tup}$
 and $\partial_j^{\hor}\Delta^{\tup}$ is the one described in the remark preceding this proof: in this case
 there are two equal contributions $[\partial_i^{\hor}\tup\colon\tup']$ and $[\partial_j^{\hor}\tup\colon\tup']$
 that cancel each other.
\end{proof}

We can filter the reduced chain complex $\tilde\Ch_*(\cms^{\infty})$ by giving filtration norm $\sum_{i=1}^lx_i$ to
the tuple $\tup=(l,\ux,\uu,\uv)$, with $\ux=(x_1,\dots,x_l)$. For example
the tuple in Figure \ref{fig:defetup} has norm 6.

By Lemma \ref{lem:doperatoropenmodtwo}
the norm is weakly decreasing along differentials. Denote by $F_p\subset\tCh_*(\cms^{\infty})$ the subcomplex generated by
tuples of norm $\leq p$, and let $F_p/F_{p-1}$ be the $p$-th filtration stratum.
Then $F_p/F_{p-1}$ is isomorphic, as a chain complex, to a direct sum of copies of $\tCh_*(C_p(\Sigma_{0,1})^{\infty})$:
there is one copy for each partition
\[
(m-p)=\sum_{i=1}^g (u_i+v_i)
\]
with $u_i,v_i\geq 0$. The isomorphism
does not preserve the degrees but shifts them by $p$.

Indeed in $F_p/F_{p-1}$ all outer
differentials vanish (see Lemma \ref{lem:doperatoropenmodtwo}): in particular the numbers $u_i,v_i$
do not change along the differentials of $F_p/F_{p-1}$. Therefore $F_p/F_{p-1}$ splits as a direct sum
of chain complexes indexed by all partitions $(m-p)=\sum_{i=1}^g (u_i+v_i)$ as above.
It is then immediate to identify the inner faces
with the ones one would have in the case $g=0$, i.e. for the surface $\Sigma_{0,1}$, after shifting
degrees by $p$.

We note that $\tCh_*(C_p(\Sigma_{0,1})^{\infty})$ is exactly the chain complex described by Fuchs in
\cite{Fuchs:CohomBraidModtwo}: we recall Fuchs' computation of the cohomology of configuration
spaces of the disc, and abbreviate $\D$ for $\mathring{\Sigma}_{0,1}$ as in Section \ref{sec:Preliminaries}.
\begin{defn}
\label{defn:symchain}
Consider a partition of $p$ into powers of 2
\[
 p=\sum_{j\geq 0}\alpha_j2^j,
\]
and let $\ualpha=(\alpha_j)_{j\geq 0}$ be the sequence of multiplicities. We understand that only finitely
many $\alpha_j$'s are strictly positive.

The associated \emph{symmetric chain} in $\tCh_*(C_p(\D)^{\infty})$, denoted by $\kappa(\ualpha)$,
is the sum of all tuples $\tup=\pa{l,(x_i)_{i\leq l}}$ such that
\begin{itemize}
 \item $l=\sum_{j\geq 0}\alpha_j$;
 \item every $x_i$ is a power of 2;
 \item for all $j\geq 0$ there are exactly $\alpha_j$ indices $i$ such that $x_i=2^j$.
\end{itemize}
A symmetric chain $\kappa(\ualpha)$ is a cycle in the chain complex $\tCh_*(C_p(\D)^{\infty})$ (see \cite{Fuchs:CohomBraidModtwo}).
We denote by $[\kappa(\ualpha)]\in\tH_{p+l}(C_p(\D)^{\infty})$ the associated homology class.
\end{defn}
In \cite{Fuchs:CohomBraidModtwo} Fuchs shows that a graded basis for $\tH_*(C_p(\D)^{\infty})$
is given by the collection of all classes $[\kappa(\ualpha)]$ associated to sequences $\ualpha=(\alpha_j)_{j\geq 0}$
which satisfy the equality $p=\sum_{j\geq 0}\alpha_j2^j$.

By Poincaré-Lefschetz duality this corresponds to a basis for $H^*(C_p(\D))$.
The \emph{dual} basis of $H_*(C_p(\D))$
happens to be the basis of monomials
\[
Q^{\ualpha}\epsilon\colon =\prod_{j\geq 0}(Q^i\epsilon)^{\alpha_j}\in H_*(C_p(\D)).
\]
This basis consists of all monomials
of weight $p$, using
the isomorphism \eqref{eq:Cohen} in its full meaning (i.e. as an isomorphism of rings,
where $Q$ denotes the first Dyer-Lashof operation).

We will not need this finer result in this article, so in the following the expression
$Q^{\ualpha}$ will only denote
the (unique) homology class in $H_{p-l}(C_p(\D))$ such that the following holds:
for all $\ualpha'=(\alpha'_j)_{j\geq 0}$ with $\sum_{j\geq 0}\alpha'_j=p$, the
algebraic intersection between $Q^{\ualpha}\epsilon\in H_{p-l}(C_p(\D))$
and $[\kappa(\ualpha')]\in\tH_{p+l}(C_p(\D)^{\infty})$
is $1\in\Z_2$ if and only $\ualpha=\ualpha'$.

We now go back to the filtered chain complex $\tCh_*(\cms^{\infty})$.
The $E^1$-page of the associated Leray spectral sequence contains on the $p$-th column
the homology of $F_p/F_{p-1}$; as we have seen the homology of this filtration stratum is
the direct sum of several copies of the homology of $\tCh_*(C_p(\D))$, one copy for each
partition $(m-p)=\sum_{i=1}^g (u_i+v_i)$ with $u_i,v_i\geq 0$.

\begin{defn}
\label{defn:gensymchain}
Consider a partition
\[
 m=p+(m-p)=\pa{\sum_{j=0}^{\infty}\alpha_j2^j}+\pa{\sum_{i=1}^g(u_i+v_i)}
\]
and let $\uu=(u_1,\dots,u_g)$, $\uv=(v_1,\dots, v_g)$, $\ualpha=(\alpha_j)_{j\geq 0}$; denote also $l=\sum_{j=0}^{\infty}\alpha_j$.

We define a chain $\kappa(p,\ualpha,\uu,\uv)$ in
$\tCh_{m+l}(\cms)$: it
is the sum of all tuples $\tup$ of the form $\pa{l,\ux,\uu,\uv}$, for varying $\ux$,
which satisfy the three properties listed in Definition \ref{defn:symchain}.

We call such a chain a
\emph{generalised symmetric chain}; we will see in the following that it is a cycle, and we will
denote by $[\kappa(p,\ualpha,\uu,\uv)]$ the associated homology class in $\tH_{m+l}(\cms^{\infty})$.

If any of $\ualpha$, $\uu$ and $\uv$ vanishes, we omit it from the notation.
\end{defn}

A generalised symmetric chain $\kappa(p,\ualpha,\uu,\uv)$
is not only
a cycle when projected to its filtration quotient $F_p/F_{p-1}$, as the $E^1$-page
of the spectral sequence tells us, but also
in the chain complex $\tCh_*(\cms^{\infty})$ itself.

To prove this fact, first note that an inner boundary of a tuple
$\tup$ preserves the norm; hence the fact that $\kappa(p,\ualpha,\uu,\uv)$ is a cycle
in $F_p/F_{p-1}$ guarantees the fact that all inner boundaries of $\kappa(p,\ualpha,\uu,\uv)$
cancel each other in $\tCh_*(\cms^{\infty})$, that is,
\[
\partial\kappa(p,\ualpha,\uu,\uv)\in \tCh_*(\cms^{\infty})
\]
is equal to the sum of all outer boundaries of the tuples involved.

Note now that outer boundaries of a generalised symmetric chain also cancel out:
the left outer boundary
of a tuple $\tup=(l,\ux,\uu,\uv)$ in the generalised symmetric chain cancels against the right outer boundary
of the tuple $\tup'=(l,\ux',\uu,\uv)$, with $x'_l=x_1$ and $x'_i=x_{i+1}$ for $1\leq i\leq l-1$. If all
$x_i$ happen to be equal, then $\tup=\tup'$ and
we are in the situation described before the proof of Lemma \ref{lem:doperatoropenmodtwo}.

The spectral sequence considered above collapses on its first page and
we have the following lemma:
\begin{lem}
\label{lem:gensymchain}
The homology $H_*(\cms^{\infty})$ has a graded basis given by the classes $[\kappa(p,\ualpha,\uu,\uv)]$
associated to generalised symmetric chains of weight $m$.
\end{lem}

\begin{defn}
 \label{defn:dualHbasis}
We can see the $\U_i$'s and $\V_i$'s as properly embedded $1$-manifolds in $\mrS$;
by Poincaré-Lefschetz duality they represent classes in $\tH_1\pa{(\mrS)^{\infty}}\simeq H^1(\S)$,
and in particular they form a basis
of the latter cohomology group. We call $\u_i,\v_i\in H_1\pa{\mrS}$ the dual basis.
\end{defn}
We establish a bijection between monomials in the tensor product of Theorem \ref{thm:Hbms*as*ggrep}
and the basis of $H^*(\cms)\simeq H_*(\cms^{\infty})$ in Lemma \ref{lem:gensymchain}:
the class 
\[
[\kappa(p,\ualpha,\uu,\uv)]\in H_{m+\sum\alpha_j}(\cms^{\infty})\simeq H^{m-\sum\alpha_j}(\cms)
\]
is associated with the monomial
\[
Q^{\ualpha}\epsilon\cdot\u^{\uu}\cdot\v^{\uv}\colon =\prod_{j=1}^{\infty}(Q^j\epsilon)^{\alpha_j}\otimes \prod_{i=1}^g(\u_i^{u_i}\v_i^{v_i}),
\]
where $\ualpha=(\alpha_j)_{j\geq 0}$, $\uu=(u_1,\dots,u_g)$ and $\uv=(v_1,\dots,v_g)$.

This shows an isomorphism of bigraded $\Z_2$-vector spaces
\begin{equation}\label{eq:isovectorspaces}
  \bigoplus_{m\geq 0} H^*(\cms)\simeq \Z_2\left[Q^j\epsilon\,|\, j\geq 0\right]\otimes\Sym_{\bullet}(\H),
\end{equation}
from which we conclude that
there exists an isomorphism as in Theorem \ref{thm:Hbms*as*ggrep}
at least \emph{as bigraded $\Z_2$-vector spaces}: the two bigraded
vector spaces have the same dimension in all bigradings.

\section{Action of \texorpdfstring{$\gg$}{Gamma(g,1)}}
\label{sec:Actiongg}
We now turn back to \emph{homology} of $\cms$. In the first subsection we describe geometrically some
homology classes, in order to give some intuition for the following subsections.
In the second subsection we prove Theorem
\ref{thm:Hbms*as*ggrep} in bigradings $(*,m)$ with $*=m$.
In the third subsection we extend the proof to all other bigradings.

\subsection{Geometric examples of homology classes}

In this subsection we consider the case $g=2$, hence $\S$ denotes the surface $\Sigma_{2,1}$.
We construct some homology classes in $H_2(C_2(\S))$: our aim is to get a first understanding
of why this homology group is isomorphic to $\Sym_2(\H)=\H^{\otimes 2}/\mathfrak{S}_2$;
the notation $\H=H_1(\S)$ was introduced in Definition \ref{defn:symplrep}.

In the following $c$ and $d$ will always denote two simple closed curves on $\S$, with corresponding homology classes $[c],[d]\in\H$.
\begin{ex}
 \label{ex:one}
Suppose that $c$ and $d$ are as in Figure \ref{fig:exampleone}: $c$ and $d$ are disjoint and non-separating.
We can consider inside $C_2(\S)$
the torus $c\times d$ of configurations in which one of the two points runs along $c$, while the other runs
along $d$. We associate to the fundamental class $\tau_1\in H_2(C_2(\S))$ of this torus the tensor product $[c]\otimes[d]\in\H^{\otimes 2}$.

Since our configurations are unordered and since we work with coefficients in $\Z_2$, the same
class $\tau_1$ can be obtained as a tensor product $[d]\otimes[c]$, i.e. exchanging the order of the two
curves: we can neglect the sign $-1$ that this operation would generate. We represent $\tau_1$ as $[c]\cdot[d]\in\Sym_2(\H)$.
\end{ex}
\begin{figure}[ht]\centering
 \includegraphics[scale=1.0]{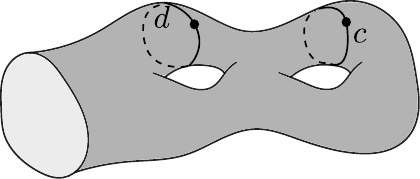}
 \caption{}
\label{fig:exampleone}
\end{figure}

\begin{ex}
 \label{ex:two}
Suppose that $c$ and $d$ are as in Figure \ref{fig:exampletwo}:
$c$ and $d$ are disjoint, and
 $c$ bounds a subsurface $\check{\S}$ of $\S$ (the shaded region), such that $d$ is not contained in $\check{\S}$.
 
 Then the homology class $\tau_2=[c]\cdot[d]$ vanishes, because the torus $c\times d$ is the boundary in $C_2(\S)$
 of the 3-manifold $\check{\S}\times d$ containing all configurations in which one point lies on $\check{\S}$ and
 the other on $d$. This is consistent with the representation $\tau_2=[c]\cdot[d]$, because $[c]=0\in \H$.
\end{ex}
\begin{figure}[ht]\centering
 \includegraphics[scale=1.0]{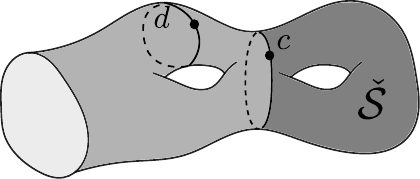}
 \caption{}
\label{fig:exampletwo}
\end{figure}

\begin{ex}
 \label{ex:three}
 Suppose that $c$ and $d$ are as in Figure \ref{fig:examplethree}: $c$ and $d$ are non-separating and 
intersect transversely in one point $P\in\S$. In this case the torus $c\times d$ is a subspace
of $\SP^2(\S)$, but $c\times d$ is not contained in $C_2(\S)$ as in Example \ref{ex:one}.

To solve this problem, we consider a small open neighborhood $P\in \U\subset\S$, and we remove from the torus $c\times d$
all configurations in which both points lie in $\U$: these removed configurations form an open disc inside the torus $c\times d$.

We obtain an embedding $\Sigma_{1,1}\hookrightarrow C_2(\S)$;
the boundary $\partial\Sigma_{1,1}$,
seen as a curve in $C_2(\S)$, is homotopic to a curve $\gamma$ in which one point spins $360^{\circ}$ around the other.
We note that the curve $\gamma\subset C_2(\S)$ is homotopic to a double covering of a curve $\gamma'\subset C_2(\S)$,
in which the two points exchange their positions after spinning $180^{\circ}$ around each other. All curves
$\partial\Sigma_{1,1}$, $\gamma$, $\gamma'$ and the homotopies relating them
are supported on the closure of $C_2(\U)$ in $C_2(\S)$.

We can therefore find a map from a M\"{o}bius band $\M$ to the closure of $C_2(\U)$ in $C_2(\S)$,
such that the images of the curves $\partial\M$ and $\partial\Sigma_{1,1}$ in $C_2(\S)$ coincide.
The union $\Sigma_{1,1}\cup_{\partial}\M$ along the boundary is then a closed
non-orientable surface of genus $3$, i.e. the connected sum of a torus and a projective plane.

The surface $\Sigma_{1,1}\cup_{\partial}\M$ is equipped with a map to $C_2(\S)$, hence its fundamental class with coefficients
in $\Z_2$ yields a homology class $\tau_3\in H_2(C_2(\S))$; thus we have managed to adapt
the construction from Example \ref{ex:one} to the case of two intersecting curves. We represent $\tau_3$
as $[c]\cdot[d]\in\Sym_2(\H)$.
\end{ex}
\begin{figure}[ht]\centering
 \includegraphics[scale=1.0]{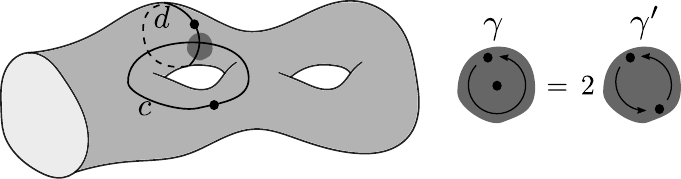}
 \caption{}
\label{fig:examplethree}
\end{figure}

\begin{ex}
 \label{ex:four}
 Suppose that $c$ and $d$ are as in Figure \ref{fig:examplefour}: $c$ and $d$ are disjoint, and
 $c$ bounds a subsurface $\check{\S}$ of $\S$ (the shaded region), such that $d$ is contained in $\check{\S}$.
 
 The torus $c\times d$ is contained in $C_2(\S)$, but it is not obvious, at first glance, why its corresponding homology
 class $\tau_4\in H_2(C_2(\S))$ should vanish: the argument used in Example \ref{ex:two} does not work immediately, because the
 3-manifold $\check{\S}\times d$ is only contained in $\SP^2(\S)$ and not in $C_2(\S)$.
 
 We try to modify this 3-manifold in the same spirit of Example \ref{ex:three}. Let $\U\tilde{\times} d\subset \check{\S}\times d$
 be an open tubular neighborhood of $d\times d$: note that this tubular neighborhood is a trivial bundle over
 $d$ with fiber an open disc $\U$. The notation $\tilde{\times}$ means that abstractly we are dealing with a product $\U\times d$,
 but not geometrically: fibers over different points of $d$ are naturally identified with different discs in $\check{\S}$.
 
 We consider the 3-manifold $\check{\S}\times d\setminus \U\tilde{\times} d$: its boundary is the disjoint union
 of the torus $c\times d$ and the torus $\partial \pa{\U\tilde{\times} d}$. The latter torus contains configurations of
 two points in $\S$, one of which spins around $d$, whereas the other is a \emph{satellite} of the first and spins
 around it $360^{\circ}$.
 
 We regard $\partial \U\tilde{\times} d$ as a trivial bundle $\gamma\tilde{\times} d$ over $d$, with fiber the curve $\gamma$
 considered in Example \ref{ex:three}. After a homotopy the latter torus
 becomes a double covering of a torus $\gamma'\tilde{\times} d$: this is a trivial bundle over $d$ with fiber the curve $\gamma'$;
 we can see $\gamma'\tilde{\times} d\subset C_2(\S)$ as a torus of configurations in which the two points
 exchange their positions spinning $180^{\circ}$ around each other, while their barycenter spins around $d$.
 
 We can fill the second boundary of the $3$-manifold $\check{\S}\times d\setminus \U\tilde{\times} d$
 with the product $\M\times d$: the output is a non-orientable $3$-manifold with boundary $c\times d$, witnessing
 that $\tau_4=0$.
 
 This is consistent with the representation $\tau_4=[c]\cdot[d]$, because $[c]=0\in \H$.
\end{ex}

\begin{figure}[ht]\centering
 \includegraphics[scale=0.8]{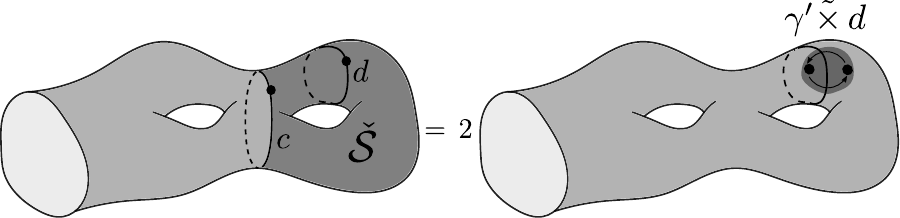}
 \caption{}
\label{fig:examplefour}
\end{figure}

\begin{ex}
 \label{ex:five}
 Suppose that $c$ and $d$ are as in Figure \ref{fig:examplefive}:
 $c$ bounds a subsurface $\check{\S}$, and $d$ is cut by $c$ into two arcs, one of which, denoted by $e$, lies
 in $\check{\S}$.
 
 In this example the complications from Examples \ref{ex:three} and \ref{ex:four} arise at the same time.
 
 To define the class $\tau_5\in H_2(C_2(\S))$, we start with the torus $c\times d$ and we perform two surgeries
 with two different M\"{o}bius bands to solve the two intersections of $c$ and $d$: the homology
 class that we obtain is represented by a non-orientable surface of genus 4, i.e. the connected
 sum of a torus and 2 projective planes.
 
 To show that $\tau_5$ vanishes, we start with the 3-manifold with boundary $\check{\S}\times d\subset\SP^2(\S)$ and we perform
 a surgery. We identify with $e$ the arc
 \[
  \check{\S}\times d\cap\pa{\SP^2(\S)\setminus C_2(\S)};
 \]
 this is a properly embedded arc in the 3-manifold with boundary $\check{\S}\times d$, and a tubular neighborhood
 of it, after a suitable homotopy, can be identified with a trivial $\U$-bundle $\U\tilde{\times} e$. We remove
 this solid cylinder from the 3-manifold and glue a trivial $\M$-bundle $\M\tilde{\times} e$, by applying an argument
 similar as the one in Examples \ref{ex:three} and \ref{ex:four}.
 
 We obtain a 3-manifold with boundary endowed with a map to $C_2(\S)$; the boundary of this 3-manifold
 is precisely the surface used to represent $\tau_5$: therefore $\tau_5=0\in H_2(C_2(\S))$,
 and this is consistent with the representation $\tau_5=[c]\cdot[d]$, because $[c]=0\in \H$.
\end{ex}

\begin{figure}[ht]\centering
 \includegraphics[scale=1]{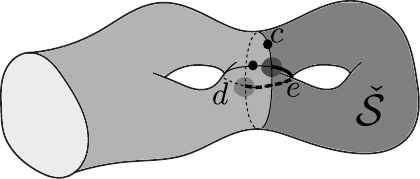}
 \caption{}
\label{fig:examplefive}
\end{figure}

If we consider more than 2 points, the examples become more and more complicated, and proving the theorem with this geometric
approach seems rather difficult.

\subsection{Bigradings of the form \texorpdfstring{$(m,m)$}{(m,m)}}
In this subsection we construct a $\gg$-~equivariant isomorphism 
\[
\psi_m\colon \Sym_m(\H))\to H_m(\cms).
\]

We already know that these $\Z_2$-vector spaces have the same dimension:
indeed $\Sym_m(\H)$ is precisely the summand in bigrading $(m,m)$ in equation
\eqref{eq:isovectorspaces}, using that a monomial
 whose
weight is \emph{equal} and not bigger than its degree cannot contain factors of the form
$Q^i\epsilon$.

The construction of $\psi_m$ is rather long and technical and involves a few definitions.

\begin{defn}
\label{defn:DinTS} 
Recall Definition \ref{defn:TS}. We denote by $\D\subset\T(\S)\cong\mrS$ the open square $]1/4;3/4[\times]1/2,1[$.
Note that $\D$ is an open disc in $\mrS$ near $\partial\S$ and 
it is disjoint from all $\overline{\U}_i$'s and $\overline{\V}_i$'s. The interior of the surface $\S'$
is then identified with $\T(\S)\setminus [1/4,3/4]\times[1/2,1[$.
\end{defn}

\begin{defn}
\label{defn:cCm}
Let $\cC^m$ be the the (discrete) set of isotopy classes of $m$-tuples of simple closed curves $c_1,\dots,c_m\subset\mrS$
such that any two curves $c_i,c_j$ intersect each other only inside $\D$.
 
Curves are seen as maps $\Sone\to\mrS$, and the intersection of two curves is the intersection of their
images. Here and in the following $\Sone$ is the unit circle in $\C$.

Two $m$-tuples of curves are isotopic if there is an ambient isotopy of $\S$ relative to $\partial\S\cup\D$
transforming one $m$-tuple into the other. In particular $\cC^m$ is more than countable.

An element of $\cC^m$ is called \emph{multicurve}; by abuse of notation the $m$-tuple
$(c_1,\dots,c_m)$ will often represent its class in $\cC^m$, i.e. the corresponding multicurve.
See picture \ref{fig:defcCm}

We denote by $\ZcC{m}$ the free $\Z_2$-vector space with basis $\cC^m$.
There is a canonical surjective map $\pr_m\colon\ZcC{m}\to\Sym_m(\H)$ given by
\[
 \pr_m(c_1,\dots,c_m)=[c_1]\cdot\ldots\cdot[c_m],
\]
where $[c_i]\in\H$ is the fundamental class of the curve $c_i$.

The group $\gg$ acts both on $\ZcC{m}$, by acting on its basis $\cC^m$, and
on $\Sym_m(\H)$, symplectically. The map $\pr_m$ is $\gg$-equivariant.
\end{defn}

\begin{figure}[ht]\centering
 \includegraphics[scale=0.7]{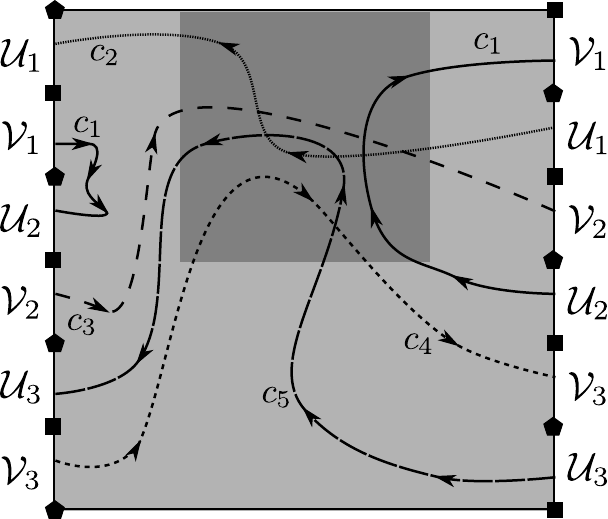}
 \caption{A multicurve in $\cC^5$.}
\label{fig:defcCm}
\end{figure}

\begin{defn}
 \label{defn:cmsD}
 Recall Definition \ref{defn:SP}.
 The space $\cmsD$ is the subspace of $\SP^m(\mrS)$ of configurations where all points
 having multiplicity $\geq 2$ lie inside $D$.
 See Figure \ref{fig:defcmsD}.
 
  Note that $\cmsD$ is open in $\SP^m(\mrS)$. There is an open inclusion $\iota_m\colon\cms\subset\cmsD$ and there is a natural map
  \[
  \j_m\colon \ZcC{m}\to H_m(\cmsD)
  \]
  defined as follows: for a class $(c_1,\dots,c_m)\in\cC^m$
  the composition
  \[
   \begin{CD}
    \pa{\Sone}^{\times m} @>c_1\times\dots\times c_m >> (\mrS)^{\times m} @>>> \SP^m(\mrS)
   \end{CD}
  \]
has image in the subspace $\cmsD$; we define $\j_m(c_1,\dots,c_m)$ as the image
of the fundamental class
of the $m$-fold torus in $H_m(\cmsD)$. The result does not change
if we substitute $(c_1,\dots,c_m)$ with another isotopic $m$-tuple of curves in the same multicurve.

We call $[c_1]\cdot\ldots\cdot[c_m]$ the image in the singular chain complex of $\cmsD$ of
the fundamental cycle of the
$m$-fold torus: this cycle represents the class $\j_m(c_1,\dots,c_m)$ and it is \emph{supported}
on the union $c_1\cup\dots\cup c_m$, meaning that it hits configurations in $\cmsD$ of points of $\mrS$
lying in this union.

The group $\Diff(\S,\D\cup\partial\S)$ acts on $\cmsD$, and there is an induced action of $\gg$
on $H_m(\cmsD)$. The map $\j_m$ is $\gg$-equivariant.
\end{defn}

\begin{lem}
 \label{lem:cms->cmsDinj}
 The inclusion $\iota_m\colon\cms\to\cmsD$ induces an injective map
 \[
  \pa{\iota_m}_*\colon H_m(\cms)\to H_m(\cmsD).
 \]
\end{lem}
\begin{proof}
 It is equivalent to prove that the map $\iota_m^*\colon H^m(\cmsD)\to H^m(\cms)$ is surjective,
 or, using Poincaré-Lefschetz duality, that the map $\tH_m(\cmsD^{\infty})\to \tH_m(\cms^{\infty})$
 is surjective: this last map is induced by the map $\cmsD^{\infty}\to\cms^{\infty}$ collapsing the subspace
 $\cmsD^{\infty}\setminus\cms$ to $\infty$.
 
 Recall Lemma \ref{lem:gensymchain} and Definition \ref{defn:gensymchain}. A basis for $\tH_m(\cms^{\infty})$
 is given by classes $[\kappa(0,\uu,\uv)]$, with
 $\uu=(u_1,\dots, u_g)$, $\uv=(v_1,\dots,v_g)$ and
 $\sum_{i=1}^g (u_i+v_i)=m$.
 
 The class $[\kappa(\uu,\uv)]$ is represented by a generalised symmetric chain consisting of only one tuple $\tup=(0,\uu,\uv)$.
 
 In particular there is a map of pairs
 $\phi^{\tup}\colon(\Delta^{\tup},\partial\Delta^{\tup})\to (\cms^{\infty},\infty)$, and the class $[\kappa(\uu,\uv)]$
 is the image along this map of the fundamental class of $H_m\pa{\Delta^{\tup},\partial\Delta^{\tup}}$.
 
 It is straightforward to check that the map $\phi^{\tup}$ factors through
 $(\cmsD,\infty)$, as a map of pairs; surjectivity of $\tH_m(\cmsD^{\infty})\to \tH_m(\cms^{\infty})$ follows.
\end{proof}

We have the following diagram of $\gg$-equivariant maps, where the full arrows are those
that we have already constructed, and we still have to prove the existence of the dashed arrows
\begin{equation}
\label{eq:fulldasheddiagram}
\begin{tikzcd}[column sep=8em,row sep=5em]
  \ZcC{m} \ar[r,"\pr_m",two heads]\ar[d,dashed, swap, "\tpsi_m"]\ar[dr,"\j_m",near start]
  & \Sym_m(\H)\ar[d,dashed ]\ar[dl,dashed,very near start,"\psi_m"]\\
  H_m(\cms)\ar[r,swap,"\pa{\iota_m}*",hook] & H_m\pa{\cmsD}.
 \end{tikzcd}
\end{equation}

We now prove that the map $\j_m$ lifts along $(\iota_m)_*$
to a $\gg$-equivariant map $\tpsi_m$ as in the diagram.
Since $(\iota_m)_*$ is injective by Lemma \ref{lem:cms->cmsDinj}, it suffices to prove that $\j_m$
lands in the image of $(\iota_m)_*$, and this last statement does not depend on how $\gg$ acts on these groups.

We will prove by induction on $m$ the following technical lemma:

\begin{lem}
 \label{lem:tpsiwithproperties}
For each $m$-tuple of curves $(c_1,\dots,c_m)$ representing a class in $\cC^m$
and for each open neighborhood $\N\subset\mrS$ of
$\D\cup c_1\cup\dots\cup c_m$, there is a singular cycle $\fc=\tpsi_m(c_1,\dots,c_m)$
in $\cms$ with the following properties:
\begin{itemize}
 \item the cycle $\fc$ is \emph{supported} on $\N$, i.e.,
this singular cycle only hits configurations of $m$ distinct points of $\mrS$ that actually lie in $\N$;
\item $\pa{\iota_m}_*(\fc)$ represents the homology class $\j_m(c_1,\dots,c_m)\in H_m\pa{\cmsD}$;
\item the two cycles $\pa{\iota_m}_*(\fc)$ and $[c_1]\cdot\ldots\cdot[c_m]$
are connected by a homology in $\cmsD$ which is supported on $\N$ (the word
\emph{homology} denotes here a $(m+1)$-singular chain whose boundary is the difference between the two cycles).
\end{itemize}
\end{lem}

For $m=0$ both $\ZcC{0}$ and $H_0(C_0(\S))$ are
isomorphic to $\Z_2$ and there is nothing to show. For $m=1$
we have a canonical identification $H_1(C_1(\S))\simeq\H\simeq\Sym_1(\H)$, so we take $\tpsi_1=\pr_1$;
obviously for all $c_1$ representing a class in $\cC^1$, the homology class $\pr_1(c_1)\in\H$
is represented by a cycle supported on $c_1$,
and in this case the cycles $\iota_*\pa{\tpsi(c_1)}$ and $\pr_1(c_1)=\j_m(c_1)$ coincide.

Let now $m\geq 1$ and in the following fix a class $(c_1,\dots, c_{m+1})\in\cC^{m+1}$.

\begin{defn}
\label{defn:variationsCm}
We introduce several variations of the notion of configuration space; see Figure \ref{fig:defcmsD}.
\begin{itemize} 
 \item The space $C_{1,m}(\S)$ is the subspace of $\mrS\times \cms$ containing all configurations
 $\pa{\bar P;\set{P_1,\dots,P_m}}$ with $\bar P\neq P_i$ for all $i$; in other words it is
 the space of configurations of $m+1$ points, one of which is \emph{white} (meaning that
 it is distinguishable from the other points), whereas the other are \emph{black} and not distinguishable
 from each other.
 \item The space $\fmstwoD$ is the subspace of $\mrS\times\cms$ containing all configurations
 $\pa{\bar P;\set{P_1,\dots, P_m}}$ where either $\bar P\in\D$ may coincide with \emph{exactly}
 one $P_i$, or
 $\bar P\not\in \D$ must be distinct from all $P_i$'s.
 Again $\bar P$ is called the white point.
 \item The space $\cmstwoD$ is the subspace of $\SP^{m+1}(\mrS)$ of configurations where either all $m+1$ points
 are distinct, or there is exactly one point \emph{inside $\D$} with multiplicity $2$ and $m-1$ other points,
 somewhere in $\mrS$, with multiplicity $1$.
\end{itemize}
 \end{defn}
 
\begin{figure}[ht]\centering
 \includegraphics{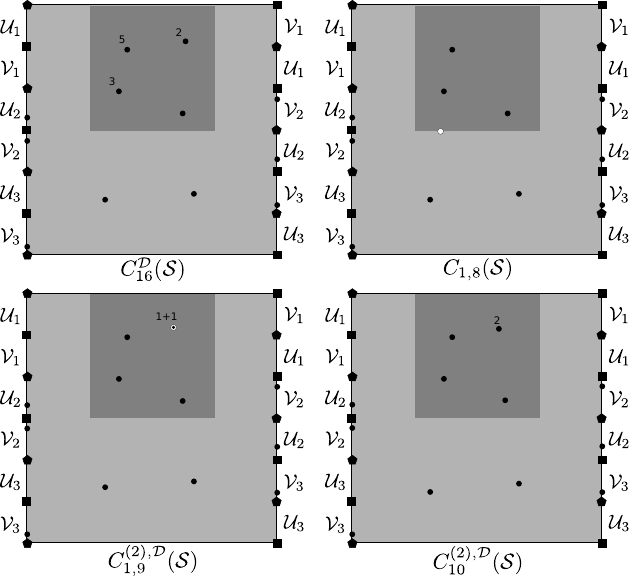}
 \caption{A configuration in each of the space introduced in Definitions \ref{defn:cmsD} and \ref{defn:variationsCm}.
 Whenever a multiplicity is not specified, it is equal to 1.}
\label{fig:defcmsD}
\end{figure}

We have the following inclusions:
\[
C_{1,m}(\S)\subset\fmstwoD\subset\mrS\times\cmsD\subset\mrS\times\SP^m(\mrS);
\]
\[
C_{m+1}(\S)\subset\cmstwoD\subset C_{m+1}^{\D}(\S)\subset\SP^{m+1}(\mrS).
\]
All these spaces are
manifolds of dimension $2m+2$ and all inclusions are open.

In particular there is a sequence of maps
\[
 H_{m+1}\pa{C_{m+1}(\S)}\to H_{m+1}\pa{\cmstwoD}\to H_{m+1}\pa{C_{m+1}^{\D}(\S)}
\]
and we will first lift the homology class $j_{m+1}(c_1,\dots,c_{m+1})$ to $H_m\pa{\cmstwoD}$ and
then to $H_{m+1}(C_{m+1}(\S))$, each time controlling the support of our representing
cycles and of the homologies between them.

Fix a neighborhood $\N$ of $\D\cup c_1\cup\dots\cup c_{m+1}$.

For the first lift, let $\N'=\pa{\N\setminus c_1}\cup\D$; note that $\N'$ is open
in $\mrS$ and contains $\D\cup c_2\cup\dots\cup c_{m+1}$, so $\N'$ is an open
neighborhood of $\D\cup c_2\cup\dots\cup c_{m+1}$ in $\mrS$. By inductive hypothesis
there is a cycle $\fc=\tpsi_m(c_2,\dots,c_{m+1})$ in $\cms$ which is supported on $\N'$,
and such that $(\iota_m)_*(\fc)$ is homologous to $[c_2]\cdot\ldots\cdot[c_{m+1}]$
along a homology in $\cmsD$ supported on $\N'$ as well.

We can multiply both
cycles and the homology between them by the cycle $[c_1]$: the result are the two homologous cycles
$[c_1]\cdot \fc$ and $[c_1]\cdot\ldots\cdot[c_{m+1}]$ in $C_{m+1}^{\D}(\S)$: both cycles and the
homology between them are supported on $\N$. Note now that the
cycle $[c_1]\cdot\fc$ lives in $\cmstwoD$, so the first lift is done and we can now
deal with the second lift.

There is a natural map $\p\colon \fmstwoD\to \cmstwoD$, which converts the white point
into a black point. This map restricts to 
a map $C_{1,m}(\S)\to C_{m+1}(\S)$, so that
we have a commutative diagram
 \begin{equation}\label{eq:cmstwodiagram}
  \begin{CD}
   C_{1,m}(\S) @>\subset >> \fmstwoD
\\   @V\p VV @V\p VV
\\   C_{m+1}(\S) @>\subset >> \cmstwoD
   \end{CD}
\end{equation}

\begin{defn}
\label{defn:falsediagonals} 
Let 
\[
\Dmone=\cmstwoD\setminus C_{m+1}(\S)
\]
and similarly
\[
\Donem=\fmstwoD\setminus C_{1,m}(\S).
\]
We note that $\p$ restricts to a homeomorphism $\Donem\to\Dmone$. Moreover both $\Dmone\subset\cmstwoD$
and $\Donem\subset\fmstwoD$ are closed submanifolds of codimension $2$, and the map $\p$ restricts to
a $2$-fold ramified covering between their respective normal bundles.
\end{defn}

Diagram \eqref{eq:cmstwodiagram} induces a commutative diagram in homology
\begin{equation}
 \label{eq:fivediagram}
\minCDarrowwidth15pt
 \begin{CD}
  @. H_{m+1}\pa{\fmstwoD} @>>> H_{m+1}\pa{\fmstwoD, C_{1,m}(\S)}\\
  @. @V\p_*VV @V\p_*VV\\
  H_{m+1}\pa{ C_{m+1}(\S)} @>>> H_{m+1}\pa{\cmstwoD} @>>> H_{m+1}\pa{\cmstwoD,C_{m+1}(\S)}
 \end{CD}
\end{equation}

Recall that we want to lift the homology class represented by the cycle $[c_1]\cdot\fc$
from the bottom central group to the bottom left group.

We first note that there is a lift of $[c_1]\cdot\fc$ to a cycle $[c_1]\otimes\fc$ in $\pa{\fmstwoD}$:
this is defined by declaring the point in $[c_1]\cdot\fc$ that spins around $c_1$
to be white. We then note that the right vertical map
\[
\p_*\colon H_{m+1}\pa{\fmstwoD, C_{1,m}(\S)} \to H_{m+1}\pa{\cmstwoD,C_{m+1}(\S)}
\]
can be rewritten, after using excision to tubular neighborhoods of $\Donem$ and $\Dmone$ respectively,
and the Thom isomorphism, as a map
\[
 H_{m-1}(\Donem)\to H_{m-1}(\Dmone).
\]
The latter map is multiplication by $2$, after identifying $\Dmone$ and $\Donem$ along $p$:
indeed the normal bundle of $\Donem$ is a
double covering of the normal bundle of $\Dmone$, hence the Thom class of the first disc
bundle corresponds to twice the Thom class of the second disc bundle. We are working
with coefficients in $\Z_2$, so multiplication by $2$ is the zero map.

Therefore the image of the cycle $[c_1]\otimes \fc$ along the diagonal of the square
in diagram \eqref{eq:fivediagram} is zero; hence the image of $[c_1]\cdot\fc$ in $H_{m+1}\pa{\cmstwoD,C_{m+1}(\S)}$
is zero; hence the homology class of $[c_1]\cdot\fc$ comes from $H_{m+1}(C_{m+1}(\S))$. More
precisely, there exists a cycle $\fc'$ in $C_{m+1}(\S)$ such that $\pa{\iota_{m+1}}_*(\fc')$ is homologous
to $[c_1]\cdot\fc$.

To prove Lemma \ref{lem:tpsiwithproperties} we need to find a good cycle and
a good homology, namely two that are supported on $\N$: a priori both $\fc'$ and the homology
between $\pa{\iota_{m+1}}_*(\fc')$ and $[c_1]\cdot\fc$ are only supported on $\mrS$.

This can be done by replacing, in the whole argument of the proof, the surface $\mrS$ with the surface $\N$.
We can define configuration spaces as in Definition \ref{defn:variationsCm} also for the open surface
$\N$, and we can repeat the argument considering $\N$ as the \emph{ambient surface}:
indeed we only needed a surface containing $\D$ and all curves $c_1,\dots,c_{m+1}$.

It is crucial that the action of $\gg$ is not involved in the statement of Lemma
\ref{lem:tpsiwithproperties}, as $\N\subset\S$ is not preserved, even up to isotopy,
by diffeomorphisms of $\S$.
Lemma \ref{lem:tpsiwithproperties} is proved.

We now have to prove the following lemma to conclude the proof of Theorem \ref{thm:Hbms*as*ggrep}
in bigradings $(m,m)$.
\begin{lem}
 \label{lem:tpsi->psi}
The map $\tpsi_m\colon\ZcC{m}\to H_m(\cms)$ is surjective and factors through the map $\pr_m$.
\end{lem}
\begin{proof}
The factorisation is equivalent to the inclusion $\ker\pr_m\subseteq\ker\tpsi_m$: since both
$\pr_m$ and $\tpsi_m$
are $\gg$-equivariant, also the induced map of vector spaces
\[
\Sym_m(\H)=\ZcC{m}/\ker\pr_m\to H_m(\cms)
\]
will automatically be
$\gg$-equivariant.

Recall from the proof of Lemma \ref{lem:cms->cmsDinj} that a basis for $\tH_m(\cms^{\infty})$ 
is given by the classes $[\kappa(\uu,\uv)]$, represented by generalised
symmetric chains consisting of only one tuple $\tup=(0,\uu,\uv)$, for some vectors
$\uu=(u_1,\dots,u_g)$ and $\uv=(v_1,\dots,v_g)$ satisfying $\sum_{i=1}^g(u_i+v_i)=m$.

The homology class
$[\kappa(\uu,\uv)]$ is the fundamental class
of the sphere $e^{\tup}\cup\set{\infty}\subset\cms^{\infty}$: the inclusion of this sphere in $\cms^{\infty}$
restricts to a proper embedding $e^{\tup}\subset\cms$.

By Poincaré Lefschetz duality $\tH_m(\cms^{\infty})\simeq H^m(\cms)$, and the latter is
the dual of $H_m(\cms)$.

We can therefore associate to $[\kappa(\uu,\uv)]$ a linear functional
on $H_m(\cms)$. This is the algebraic intersection product with the cell $e^{\tup}$, seen as a proper submanifold of $\cms$:
we denote it by
\[
 \cdot\cap e^{\tup}\colon H_m(\cms)\to\Z_2.
\]
Therefore
\[
 \ker\tpsi_m=\bigcap_{\tup}\ker\pa{(\cdot\cap e^{\tup})\circ\tpsi_m},
\]
and it suffices to check that $\ker\pr_m\subseteq\ker\pa{(\cdot\cap e^{\tup})\circ\tpsi_m}$
for all $\tup$ of the form $(0,\uu,\uv)$, or equivalently, that $(\cdot\cap e^{\tup})\circ\tpsi_m$ factors through $\pr_m$.

Recall from the proof of Lemma \ref{lem:cms->cmsDinj} that the cohomology class $(\cdot\cap e^{\tup})$ on $\cms$ is a pullback
of a cohomology class of $\cmsD$, that we call $(\cdot\cap e^{\tup})^{\D}$. Alternatively,
note that the inclusion $e^{\tup}\cap \set{\infty}\to\cms^{\infty}$ is the composition of the inclusion $e^{\tup}\cap\set{\infty}\to\cmsD^{\infty}$
and the quotient map $\cmsD^{\infty}\to\cms^{\infty}$, and consider the fundamental class of the sphere $e^{\tup}\cup\set{\infty}$ and its images.

We can therefore compute the map $(\cdot\cap e^{\tup})\circ\tpsi_m$ as the map 
\begin{equation}
\label{eq:badformula}
(\cdot\cap e^{\tup})^{\D}\circ\j_m\colon \ZcC{m}\to\Z_2.
\end{equation}
The latter map coincides with the composition
\begin{equation}
 \label{eq:productformula}
\pa{ \prod_{i=1}^g(\cdot\cap\U_i)^{u_i}(\cdot\cap\V_i)^{v_i}}\circ\pr_m,
\end{equation}

where $\prod_{i=1}^g(\cdot\cap\U_i)^{u_i}(\cdot\cap\V_i)^{v_i}\in\Sym_m(\Hom(\H;\Z_2))=\Hom\pa{\Sym_m(\H);\Z_2}$.

This can be checked on every generator $(c_1,\dots, c_m)\in\ZcC{m}$ by chosing in the isotopy class a representative $(c_1,\dots, c_m)$
with all curves $c_i$ transverse to all segments $\U_j$ and $\V_j$.

Consider again the map $c_1\times\dots\times c_m\colon \pa{\Sone}^{\times m}\to\cmsD$ that we used to define the cycle $[c_1]\cdot\dots[c_m]$
representing the class $\j_m(c_1,\dots,c_m)$ (see Definition \ref{defn:cmsD}):
this map is an embedding near $e^{\tup}$ and is transverse to $e^{\tup}$.

The equality of the maps in equations \eqref{eq:badformula} and \eqref{eq:productformula} on the generator $(c_1,\dots,c_m)\in\ZcC{m}$
follows from a straightforward
computation (in $\Z_2$) of the cardinality of the set $[c_1]\cdot\dots[c_m]\cap e^{\tup}$ in terms
of the cardinalities of all sets of the form $c_i\cap\U_j$ and $c_i\cap \V_j$.
In particular $(\cdot\cap e^{\tup})\circ\tpsi_m$ factors through $\pr_m$.

To show surjectivity of $\tpsi_m$, choose a tuple $\tup$ of the form $(0,\uu,\uv)$
and an $m$-tuple of curves $(c_1,\dots,c_m)$ containing, for every $1\leq i\leq g$, $u_i$ parallel
copies of some curve representing $\u_i$ and $v_i$ parallel copies of some curve representing
$\v_i$ (see Definition \ref{defn:dualHbasis}), such that all intersections between these curves
lie in $\D$.

Then $j_m(c_1,\dots,c_m)\cap e^{\tup}=1$ and for all
other tuples $\tup'$ of the form $(0,\uu',\uv)$ we have instead $j_m(c_1,\dots,c_m)\cap e^{\tup'}=0$.

This shows that $\psi_m(c_1,\dots,c_m)=[c_1]\cdot\ldots\cdot[c_m]\in\Sym_m(\H)$, which is
one of the generating monomials.
\end{proof}

Theorem \ref{thm:Hbms*as*ggrep} is now proved in all bigradings of the form $(m,m)$.

\subsection{General bigradings \texorpdfstring{$(m-l,m)$}{(m-l,m)}.} Fix $0\leq l\leq m$ for the whole subsection: our next aim is to prove Theorem
\ref{thm:Hbms*as*ggrep} for the bigrading $(m-l,m)$.

For all $0\leq p\leq m$, the group $\Diff(\S;\partial\S\cup\D)$ acts both on $C_p(\D)\times C_{m-p}(\S')$
and on $\cms$, and the map $\mu$ is equivariant with respect to this action
(see Definition \ref{defn:universalSbundle});
hence, using the K\"{u}nneth formula, there is an induced $\gg$-equivariant map in homology
\begin{equation}
 \label{eq:mu*}
 \mu_*\colon H_{p-l}\pa{C_p(\D)}\otimes H_{m-p}\pa{C_{m-p}(\S')}\to H_{m-l}\pa{\cms}.
\end{equation}

Note that $H_{p-l}\pa{C_p(\D)}\otimes H_{m-p}\pa{C_{m-p}(\S')}$
is the tensor product of the trivial representation $H_{p-l}\pa{C_p(\D)}$, and
of the representation $H_{m-p}\pa{C_{m-p}(\S')}$, which by the results of the previous
section is isomorphic to the symplectic representation $\Sym_{m-p}(\H)$.

We will prove the following lemma, from which Theorem \ref{thm:Hbms*as*ggrep} follows:
\begin{lem}
 \label{lem:oplussplitting}
For all $l\leq p\leq m$ the map $\mu_*$ in equation \eqref{eq:mu*} is injective, and the collection
of all these maps yields a splitting
\begin{equation}
  \label{eq:musplitting}
 H_{m-l}(\cms)=\bigoplus_{p=l}^m H_{p-l}\pa{C_p(\D)}\otimes H_{m-p}\pa{C_{m-p}(\S')}.
\end{equation}
\end{lem}
\begin{proof}
Note that the statement of the lemma does not depend on the the action of $\gg$:
we have a map from the right-hand side to the left-hand side of equation \eqref{eq:musplitting},
we already know that it is $\gg$-equivariant,
we only need to show that it is a linear isomorphism. Note also that Lemma \ref{lem:gensymchain}
implies that the two vector spaces have the same dimension.

Fix $l\leq p\leq m$ and $\ualpha=(\alpha_j)_{j\geq 0}$, and let $[a]=Q^{\ualpha}\epsilon=\prod_{j=0}^{\infty}(Q^j\epsilon)^{\alpha_j}$ be a generator of
$H_{p-l}(C_p(\D))$, hence $l=\sum_j\alpha_j$ and $p=\sum_j\alpha_j2^j$.

Fix also $\uu=(u_1,\dots, u_g)$ and $\uv=(v_1,\dots,v_g)$,
and let $[b]=\u^{\uu}\cdot\v^{\uv}=\prod_{i=1}^g (\u_i^{u_i}\v_i^{v_i})$ be a generator of $H_{m-p}(C_{m-p}(\S'))$, using the
isomorphism proved in the previous subsection, with $(m-p)=\sum_i(u_i+v_i)$.

Here $a$ and
$b$ are chosen singular cycles representing the homology classes, with $a$ supported on $\D$
and $b$ supported on $\S'$. See Figure \ref{fig:axb}

\begin{figure}[ht]\centering
 \includegraphics[scale=0.65]{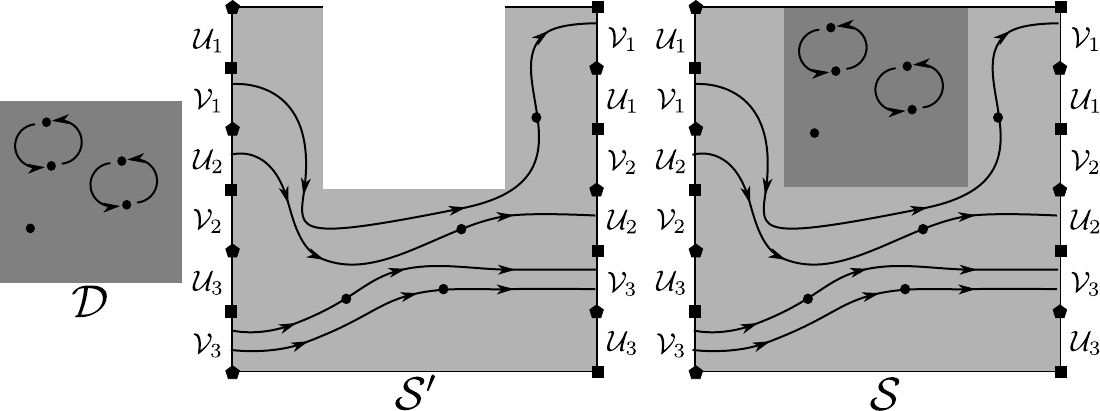}
 \caption{From left to right, the class $[a]=\epsilon\cdot(Q\epsilon)^2\in H_2(C_5(\D))$;
 the class $[b]=\v_1\cdot\u_2\cdot\v_3^2\in H_4(C_4(\S'))$; and the product class
 $\mu_*([a]\otimes[b])\in H_6(C_9(\S))$.}
\label{fig:axb}
\end{figure}

Then $[a]\otimes [b]$ is a generator of $H_{p-l}(C_p(\D))\otimes H_{m-p}(C_{m-p}(\S'))$,
by the K\"unneth formula, and we are interested in the homology class $\mu_*([a]\otimes [b])$.

There is one such
class for any choice of $[a]$ and $[b]$ as above, that is, for any choice of
$p$, $\ualpha$, $\uu$ and $\uv$ satisfying the conditions
$l=\sum_{j\geq 0}\alpha_j$, $p=\sum_{j\geq 0}\alpha_j2^j$ and $(m-p)=\sum_{i=1}^g(u_i+v_i)$, where
we use the notation above.

We want to show that the collection of all the corresponding classes of the form $\mu_*([a]\otimes [b])$
gives a basis for $H_{m-l}(\cms)$.

We will study
the intersection of $\mu_*([a]\otimes [b])$ with cohomology classes of $\cms$ represented by
generalised symmetric chains in $\cms^{\infty}$.

To compute the algebraic intersection
between $\mu_*([a]\otimes [b])$ and $[\kappa(p',\ualpha',\uu',\uv')]$ we
consider the map
\[
 \mu^{\infty}\colon \cms^{\infty}\to \pa{C_p(\D)\times C_{m-p}(\S')}^{\infty}
\]
which collapses to $\infty$ the complement in $\cms^{\infty}$ of the open submanifold $C_p(\D)\times C_{m-p}(\S')$.

By Poincaré-Lefschetz duality, the map $\mu^{\infty}_*$ in reduced homology corresponds to the cohomology map
\[
 \mu^*\colon H^*(\cms)\to H^*(C_p(\D)\times C_{m-p}(\S'))=H^*(C_p(\D))\otimes H^*(C_{m-p}(\S')).
\]

We give $\pa{C_p(\D)\times C_{m-p}(\S')}^{\infty}$ the cell complex structure of the smash product
$C_p(\D)^{\infty}\wedge C_{m-p}(\S')^{\infty}$. Here $C_p(\D)^{\infty}$ is given the cell structure of $C_p((0,1)^2)^{\infty}$
coming from the natural identification $\D=]1/4,3/4[\times]1/2,1[\cong]0,1[^2$, which is obtained by rescaling
and translating. Moreover we choose any diffeomorphism $\mrS'\cong\mrS$ that restricts to the identity
on all $\U_i$'s and $\V_i$'s, and give $C_{m-p}(\S')^{\infty}$ the cell structure of $C_{m-p}(\S)^{\infty}$.

Recall that $\cms^{\infty}$ can be filtered according to the norm of cells: a cell $e^{\tup}$ associated with
the tuple $\tup=(l,\ux,\uu,\uv)$ has norm $\sum_{i=1}^l x_i$, and the norm is weakly decreasing along boundaries.
In the previous section we just considered
the associated filtration of the reduced chain complex $\tCh_*(\cms^{\infty})$, whereas now we
consider the closed subcomplex $F_p\cms^{\infty}\subset\cms^{\infty}$, which is the union
of all cells of norm $\leq p$.

The crucial observation is that $\mu^{\infty}$ restricts to a cellular map 
\[
F_p\cms^{\infty}\to \pa{C_p(\D)\times C_{m-p}(\S')}^{\infty}.
\]

To see this, fix a tuple $\tilde{\tup}=(\tilde{l}, \tilde{\ux},\tilde{\uu},\tilde{\uv})$ of norm $\tilde{p}\leq p$
and of dimension $\tilde{l}+m$, and
consider the open cell
cell $e^{\tilde{\tup}}\subset F_p\cms^{\infty}$.

If $\tilde{p}<p$, then
$e^{\tilde{\tup}}\cap (C_p(\D)\times C_{m-p}(\S'))$ is empty. If $\tilde{p}=p$, then
\[
e^{\tilde{\tup}}\cap (C_p(\D)\times C_{m-p}(\S'))=e^{\tup'}\times e^{\tup''},
\]
where $\tup'=(\tilde{l},\tilde{\ux})$ and $\tup''=(0,\tilde{\uu},\tilde{\uv})$.

Therefore
$\mu^{\infty}(e^{\tilde{\tup}})$ is $\set{\infty}$ in the first case, and in the second case it
is contained in the union $\set{\infty}\cup e^{\tup'}\times e^{\tup''}$, which is also
contained in the $(\tilde{l}+m)$-skeleton of $\pa{C_p(\D)\times C_{m-p}(\S')}^{\infty}$.

Consider now the generalised symmetric chain $\kappa(p',\ualpha',\uu',\uv')$ representing
a class in $\tH_{m+l}(\cms^{\infty})=H^{m-l}(\cms)$, with $\ualpha'=(\alpha'_j)_{j\geq 0}$,
$\uu'=(u_1,\dots,u'_g)$ and $\uv'=(v'_1,\dots,v'_g)$; in particular $l=\sum_{j\geq 0}\alpha'_j$.
Suppose moreover $p'\leq p$.

If $p'<p$, the previous argument shows that $\mu^{\infty}_*\pa{\kappa(p',\ualpha',\uu',\uv')}=0$
in the reduced cellular chain complex,
and in particular the corresponding homology class is mapped to zero.

Suppose now $p'=p$: then the previous argument shows that
the homology class $[\kappa(p',\ualpha',\uu',\uv')]\in \tH_{m+l}(\cms^{\infty})$
is mapped along $\mu^{\infty}_*$ to the class
\[
[\kappa(\ualpha')]\otimes [\kappa(\uu',\uv')]\in\tH\pa{C_p(\D)^{\infty}\wedge C_{m-p}(\S')^{\infty}}.
\]
Indeed each tuple $\tup$ in the cycle $\kappa(p',\ualpha',\uu',\uv')|$ is mapped by $\mu^{\infty}_*$
to a corresponding pair of tuples $\tup'\otimes\tup''$ in the cycle $\kappa(\ualpha')\otimes \kappa(\uu',\uv')$,
so even at the level of chains we have
\[
\mu^{\infty}_*\pa{\kappa(p',\ualpha',\uu',\uv')}=\kappa(\ualpha')\otimes \kappa(\uu',\uv').
\]

We can now compute the algebraic intersection of $\mu_*([a]\otimes [b])$ with
the cohomology class $[\kappa(p',\ualpha',\uu',\uv')]$ as the algebraic intersection between
$[a]\otimes [b]$ and $\mu^{\infty}_*\pa{[\kappa(p',\ualpha',\uu',\uv')]}$.

For $p'<p$ the previous argument show that this intersection is zero.

For $p'=p$ the intersection between
$[a]\otimes [b]$ and $\mu^{\infty}_*([\kappa(p,\ualpha',\uu',\uv')])=[\kappa(\ualpha')]\otimes [\kappa(\uu',\uv')]$
is $1\in\Z_2$ exactly when $\ualpha=\ualpha'$, $\uu=\uu'$ and $\uv=\uv'$; otherwise it is $0$.

To finish the proof we consider the collection of all strings of the form
\[
\pa{p,\ualpha=(\alpha_j)_{j\geq 0},\uu=(u_1,\dots,u_g),\uv=(v_1,\dots,v_g)}
\]
satisfying $l=\sum_j\alpha_j$,$p=\sum_j \alpha_j2^j$ and $(m-p)=\sum_i(u_i+v_i)$;
we choose a total order on the set of these strings,
such that the parameter $p$ is weakly increasing along this order; we associate
to each string its corresponding class in $H_{m-l}(\cms)$ of the form $\mu_*([a]\otimes [b])$ and its
corresponding class $[\kappa(p,\ualpha',\uu',\uv')]\in\tH_{m+l}(\cms^{\infty})$.

Then the matrix of algebraic intersections between these two sets of
classes is an upper-triangular matrix
with $1$'s on the diagonal, and in particular
it is invertible. This shows that the set of classes of the form $\mu_*([a]\otimes [b])$
is a basis for $H_{m-l}(\cms)$.
\end{proof}

One could expect that the basis given by classes of the form $[a]\otimes [b]\in H_{m-l}(\cms)$
is also \emph{dual} to the basis of classes $[\kappa(p,\ualpha,\uu,\uv)]\in \tH_{m+l}(\cms^{\infty})$, i.e., the
matrix considered in the end of the previous proof is not only upper-triangular but also
diagonal. This is however not true, as the following example shows.

Let $g=1$, $m=2$, $p=1$, $p'=2$ and consider the classes $[a]=\epsilon\in H_0(C_1(\D))$,
$[b]=\u_1\in \H=H_1(C_1(\S'))$. Moreover let the generalised symmetric chain
$\kappa(p',\ualpha',\uu',\uv')$ be defined by
$\ualpha'=(\alpha'_j)_{j\geq 0}$ with $\alpha'_1=1$ and all other $\alpha_j=0$, $\uu'=(u'_1=0)$
and $\uv'=(v'_1=0)$.

Represent $[a]$ by a point in $a\in\D$, for example
the point $(1/2,3/4)$; represent $[b]$ by a simple closed curve $b\subset\S'$
that intersects only once, transversely, the vertical segment passing through $a$, i.e.
$\set{1/2}\times]0,1[$. See Figure \ref{fig:counterexample}.

\begin{figure}[ht]\centering
 \includegraphics[scale=0.6]{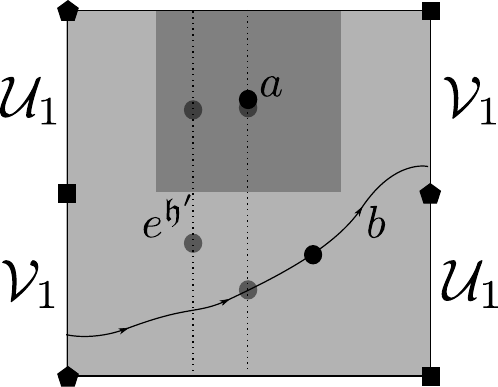}
 \caption{$\mu_*(a\otimes b)$ and $e^{\tup'}$ intersect once, transversely inside $C_2(\S)$.}
\label{fig:counterexample}
\end{figure}

Then the cycle $\kappa(p',\ualpha',\uu',\uv')$ consists uniquely of one tuple
\[
\tup'=(1,\ux'=(x'_1=2),\uu'=(u'_1=0),\uv'=(v'_1=0)).
\]
The corresponding cell $e^{\tup'}$ intersects once, transversely,
the cycle $\mu_*(a\otimes b)$, which is represented by the curve of configurations of two
points in $\mrS$, one of which is fixed at $a$ whereas the other runs along $b$:
there is exactly one position on $b$ lying under $a$.

Hence the algebraic intersection between these two classes is $1$, and since $p'>p$
this is an entry strictly above the diagonal in the matrix considered in the proof
of Lemma \ref{lem:oplussplitting}.

The proof of Theorem \ref{thm:Hbms*as*ggrep} can be easily generalised to surfaces
with more than one boundary curve. Let $\Sigma_{g,n}$ be a surface of genus
$g$ with $n\geq 1$ parametrised boundary curves and let $\Gamma_{g,n}$ be
the group of connected components of the topological group
$\Diff(\Sigma_{g,n};\partial\Sigma_{g,n})$:
then there is an isomorphism of bigraded $\Z_2$-representations
\begin{equation}
\label{eq:sgn}
\bigoplus_{m\geq 0}H_*\pa{C_m(\Sigma_{g,n})}\simeq \Z_2\left[Q^j\epsilon\,|\, j\geq 0\right]\otimes\Sym_{\bullet}(H_1(\Sigma_{g,n})),
\end{equation}

where the action of $\Gamma_{g,n}$ on the right-hand side is induced by the natural
action on $H_1(\Sigma_{g,n})$.

For $n\geq 2$ the intersection form on the vector space $H_1(\Sigma_{g,n})$
is \emph{degenerate}, but it is still invariant under the action of
$\Gamma_{g,n}$, so there is still a map from $\Gamma_{g,n}$ to the subgroup
of $GL_{2g+n-1}(\Z_2)$ fixing this bilinear form, and in this sense we can say
that the representation in \eqref{eq:sgn} is \emph{symplectic}.

The proof of the isomorphism \eqref{eq:sgn} is almost verbatim the same; the main difference is in the construction of the
model $\T(\Sigma_{g,n})$ for $\mathring{\Sigma}_{g,n}$: we divide the
vertical segments $\set{0,1}\times[0,1]\subset[0,1]^2$
into $2g+n-1$ equal parts, that we call $I_i^l$ and $I_i^r$ according to their order;
we identify, for each $i>2g$, the interval $I_i^l$ with the interval $I_i^r$;
the other couples of intervals, yielding the genus, are identified just as before.

One can further generalise to non-orientable surfaces with non-empty boundary: it suffices,
in the above construction, to glue some of the intervals $I_i^l$ and $I_i^r$ reversing their
orientation. We leave all details of these generalisations to the interested reader.

\section{Proof of Theorem \ref{thm:main}}
\label{sec:sseq}
We will prove Theorem \ref{thm:main} by induction on $m$.
The case $m=0$ is trivial.

For all $m\geq 0$ we let $E(m)$ be the Leray-Serre spectral sequence associated with
the bundle \eqref{eq:BirmanbundleD}: its second page has the form
\[
 E(m)^2_{k,q}=H_k(B\Diff(\S;\partial\S\cup\D);H_q(\cms))=H_k(\gg;H_q(\cms)).
\]
From Theorem \ref{thm:Hbms*as*ggrep} we know that this spectral sequence is concentrated
on the rows $q=0,\dots, m$.

We want to prove the vanishing of all differentials appearing in the pages
$E(m)^r$ with $r\geq 2$; the $r$-th differential takes the form
\[
 \partial_r\colon E(m)^r_{k,q}\to E(m)^r_{k-r,q+r-1}.
\]
In particular any differential $\partial_r$ exiting from the row $q=m$ is trivial,
because it lands in a higher, hence trivial row.

Fix now $q=m-l<m$, in particular $l\geq 1$; by Theorem \ref{thm:Hbms*as*ggrep}, and
in particular by Lemma \ref{lem:oplussplitting}, we have a splitting of $H_k(\gg;H_{m-l}(\cms))$ as
\begin{equation}
\label{eq:Hksplitting}
 \bigoplus_{p=l}^m H_k\pa{\gg;\mu_*\pa{H_{p-l}(C_p(\D))\otimes H_{m-p}(C_{m-p}(\S'))}}.
\end{equation} 
We fix now $l\leq p\leq m$ and show the vanishing of all differentials $\partial_r$ exiting from the
summand with label $p$ in the previous equation.

Consider the map $\mu^{\cF}$ from Definition \ref{defn:muF} as a map of
bundles over the space $B\Diff(\S;\partial\S\cup\D)$:
\[
 \mu^{\cF}\colon C_p(\D)\times C_{m-p}(\cF_{\S'})\to C_m(\cF_{\S,\D}).
\]
Note that the first bundle $C_p(\D)\times C_{m-p}(\cF_{\S'})\to B\Diff(\S;\partial\S\cup\D)$
is the product of the space $C_p(\D)$ with the bundle $C_{m-p}(\cF_{\S'})\to B\Diff(\S;\partial\S\cup\D)$;
therefore the spectral sequence associated with $C_p(\D)\times C_{m-p}(\cF_{\S'})$ is isomorphic,
from the second page on, to the tensor product
of $H_*(C_p(\D))$ and the spectral sequence associated with the bundle $C_{m-p}(\cF_{\S'})$; the
latter spectral sequence is isomorphic, in our notation, to the spectral sequence $E(m-p)$.
In particular $\mu^{\cF}$ induces a map of spectral
sequences
\[
\mu^{\cF}_*\colon H_*(C_p(\D))\otimes E(m-p)\to E(m);
\]
that in the second page, on the $(m-l)$-th row and $k$-th column, restricts to
the inclusion of one of the direct summands in equation \eqref{eq:Hksplitting}:
\[
H_k\pa{\gg;\mu_*(H_{p-l}(C_p(\D))\otimes H_{m-p}(C_{m-p}(\S')))}\subset H_k\pa{\gg;H_{m-l}(\cms)}.
\]

In particular if we prove the vanishing of all differentials $\partial_r$
in the first spectral sequence, then also all differentials
$\partial_r$ exiting from this direct summand in the
second spectral sequence $E(m)$ must vanish.

The differentials in the spectral sequence $H_*(C_p(\D))\otimes E(m-p)$
are obtained by tensoring the identity of $H_*(C_p(\D))$ with the differentials
of the spectral sequence $E(m-p)$; as $p\geq l\geq 1$ we know by inductive hypothesis
that the latter vanish. Theorem \ref{thm:main} is proved.

One can generalise Theorem \ref{thm:main} to orientable or non-orientable
surfaces with non-empty boundary, following the generalisation
of Theorem \ref{thm:Hbms*as*ggrep} discussed at the end of Section \ref{sec:Actiongg}.

\section{Comparison with the work of B\"{o}digheimer and Tillmann}
\label{sec:comparison}
In this section we will compare Theorems \ref{thm:main} and \ref{thm:Hbms*as*ggrep}
with the results in \cite{BoT}. In particular we consider the following two theorems, that
we formulate in an equivalent, but different way as in \cite{BoT}, fitting in our
framework. In the entire section homology is taken with coefficients in $\Z_2$ unless explicitly
stated otherwise.
\begin{thm}[Corollary 1.2 in \cite{BoT}]
\label{thm:BoTone}
Let $\F$ be a field, and let $m\geq 0$, $g\geq 0$. Then the following graded vector spaces are isomorphic in degrees $*\leq \frac 23 g-\frac 23$
\[
 H_*\pa{\ggm;\F}\cong H_*\pa{\gg;\F}\otimes_{\F} H_*\pa{\mathfrak{S}_m;\F[x_1,\dots,x_m]}.
\]
Here each variable $x_1,\dots,x_m$ has degree 2 and the symmetric group $\mathfrak{S}_m$ acts on the polynomial
ring $\F[x_1,\dots,x_m]$ by permuting the variables.
\end{thm}
We point out that the range of degrees $*\leq \frac 23 g-\frac 23$ is the stable range for the homology $H_*(\gg;\F)$
of the mapping class group, see \cite{Harer} for the original stability theorem and \cite{Boldsen, ORW:resolutions_homstab} for the improved
stability range.
\begin{thm}[Theorem 1.3 in \cite{BoT}]
\label{thm:BoTtwo}
For all $m\geq 0$ the map $\mu^{\cF}\colon C_1(\D)\times C_m(\cF_{\S'})\to C_{m+1}(\cF_{\S,\D})$ from Definition \ref{defn:muF}
induces a split-injective map in homology
\[
 \mu^{\cF}_*\colon H_*\pa{C_m(\cF_{\S'})}\cong H_0\pa{C_1(\D)}\otimes_{\F} H_*\pa{C_m(\cF_{\S'})}\hookrightarrow H_*\pa{C_{m+1}(\cF_{\S,\D})}.
\]
In other words, the inclusion of groups $\ggm\hookrightarrow\gg^{m+1}$ given by adding a puncture near the boundary induces a split
injective map $H_*(\ggm;\F)\hookrightarrow H_*(\gg^{m+1};\F)$.
\end{thm}

We first reformulate Theorems \ref{thm:main} and \ref{thm:Hbms*as*ggrep} in a convenient way.
\begin{defn}
 \label{defn:Cbullet}
 We introduce some abbreviations for the following disjoint unions:
 \[
 C_{\bullet}(\D)=\coprod_{m\geq 0}C_m(\D);
 \]
 \[
 C_{\bullet}(\S)=\coprod_{m\geq 0}C_m(\S); 
 \]
 \[
 C_{\bullet}(\cF_{\S,\D})=\coprod_{m\geq 0}C_m(\cF_{\S,\D}).
 \]
\end{defn}
Then $C_{\bullet}(\D)$ is a (homotopy associative) topological monoid, with associated
Pontryagin ring $H_*\pa{C_{\bullet}(\D)}\cong\Z_2[Q^j\epsilon|j\geq 0]$ (see Equation \ref{eq:Cohen}).
This is a bigraded ring, where the two gradings are the homological degree and the weight.

The maps $\mu$ from Definition \ref{defn:muF} make $C_{\bullet}(\S)$ into a (homotopy associative)
module over the monoid $C_{\bullet}(\D)$; correspondingly $H_*\pa{C_{\bullet}(\S)}$ is a bigraded
module over $H_*\pa{C_{\bullet}(\D)}$. The structure of this module is described by the following
reformulation of Theorem \ref{thm:Hbms*as*ggrep}, see in particular the proof of Lemma \ref{lem:oplussplitting}.

\begin{thm}
 \label{thm:firstreformulation}
 There is an isomorphism of $\gg$-representations in (bigraded) modules over the ring $H_*\pa{C_{\bullet}(\D)}$
 \[
  H_*\pa{C_{\bullet}(\S)}\cong   H_*\pa{C_{\bullet}(\D)}\otimes\Sym_{\bullet}(\H).
 \]
\end{thm}

Similarly, Theorem \ref{thm:main} can be reformulated as follows, considering at the same time all values of $m\geq 0$.
\begin{thm}
 \label{thm:mainreformulation}
 There is an isomorphism of (bigraded) modules over $H_*\pa{C_{\bullet}(\D)}$
 \[
 H_*\pa{C_{\bullet}(\cF_{\S,\D})}\cong   H_*\pa{C_{\bullet}(\D)}\otimes H_*\pa{\gg;\Sym_{\bullet}(\H)}.
 \] 
\end{thm}
The proof follows from the arguments used in Section \ref{sec:sseq}: we have actually shown that the direct
sum of the spectral sequences $\bigoplus_{m\geq 0} E(m)$ is itself a free module over $H_*\pa{C_{\bullet}(\D)}$
on the second page, with the same description as above.

By virtue of homological stability in $g$ for the sequences of groups $\pa{\gg}_{g\geq 0}$ and 
$\pa{\ggm}_{g\geq 0}$, Theorem \ref{thm:BoTone} can be equivalently rephrased as an isomorphism of graded vector spaces
\[
  H_*\pa{\Gamma_{\infty,1}^m;\F}\cong H_*\pa{\Gamma_{\infty,1};\F}\otimes_{\F} H_*\pa{\mathfrak{S}_m;\F[x_1,\dots,x_m]}.
\]
Here $\Gamma_{\infty,1}=\colim_{g\to\infty}\gg$ and $\Gamma_{\infty,1}^m=\colim_{g\to\infty}\ggm$.

Let $\Gamma_{\infty,1}^{\bullet}$ denote the disjoint union $\coprod_{m\geq 0}\Gamma_{\infty,1}^m$, which we
consider as a groupoid. Then Theorem \ref{thm:BoTone} is equivalent to the isomorphism of bigraded vector spaces
\[
  H_*\pa{\Gamma_{\infty,1}^{\bullet};\F}\cong H_*\pa{\Gamma_{\infty,1};\F}\otimes_{\F} \pa{\bigoplus_{m\geq 0} H_*\pa{\mathfrak{S}_m;\F[x_1,\dots,x_m]}}.
\]
We consider $\bigoplus_{m\geq 0} H_*\pa{\mathfrak{S}_m;\F[x_1,\dots,x_m]}$ as a free module over the Pontryagin ring
$H_*(\fS_{\bullet};\F)$, where $\fS_{\bullet}=\coprod_{m\geq 0}\fS_m$ and the Pontryagin product is induced by
the natural maps of groups $\fS_m\times\fS_{m'}\to\fS_{m+m'}$. This module structure is induced by the natural maps
\[
H_*\pa{\fS_m;\F[x_1,\dots,x_m]}\otimes_{\F} H_*\pa{\fS_{m'};\F}\to  H_*\pa{\fS_m\times\fS_{m'};\F[x_1,\dots,x_m]\otimes_{\F}\F} \to
\]
\[
\to H_*\pa{\fS_{m+m'};\F[x_1,\dots,x_{m+m'}]},
\]
where the first arrow is given by the homology cross product and the second by the aforementioned inclusion
$\fS_m\times\fS_{m'}\to\fS_{m+m'}$ together with the (equivariant) inclusion
of coefficients
\[
\F[x_1,\dots,x_m]\otimes_{\F}\F=\F[x_1,\dots,x_m]\subset \F[x_1,\dots,x_{m+m'}].
\]
From now on we assume $\F=\Z_2$.
Let $\beta_m=\beta_m(\D)$ denote the braid group on $m$ strands of the disc, and let $\beta_{\bullet}=\coprod_{m\geq 0}\beta_m$.
Then the natural projections $\beta_m\to\fS_m$ induce an \emph{inclusion} of Pontryagin rings
\[
 H_*(C_{\bullet}(\D))\cong H_*(\beta_{\bullet})\hookrightarrow H_*(\fS_{\bullet}).
\]
See again the appendix to \cite[Chap.III]{CLM} for a computation of $H_*(\beta_{\bullet})$,
and \cite{Nakaoka:infinite} or \cite{GiustiSalvatoreSinha} for a computation of $H_*(\fS_{\bullet})$.

Hence $\bigoplus_{m\geq 0} H_*\pa{\mathfrak{S}_m;\Z_2[x_1,\dots,x_m]}$ becomes a free module over $H_*(C_{\bullet}(\D))$.
Theorem \ref{thm:mainreformulation} gives then the following isomorphism, which turns out to be an isomorphism of free
$H_*(C_{\bullet}(\D))$-modules.
\begin{equation}
 \label{eq:doubleiso}
\begin{split}
  H_*\pa{\Gamma_{\infty,1}^{\bullet}} & \cong H_*(C_{\bullet}(\D))\otimes H_*\pa{\Gamma_{\infty,1};\Sym_{\bullet}(\H_{\infty})}\\
  & \cong H_*\pa{\Gamma_{\infty,1}}\otimes \pa{\bigoplus_{m\geq 0} H_*\pa{\mathfrak{S}_m;\Z_2[x_1,\dots,x_m]}}.
  \end{split}
\end{equation}
Here $\H_{\infty}=\colim_{g\to\infty} H_1(\Sigma_{g,1})$ is the first homology of the surface of infinite genus with one
boundary component.

Recall that $H_*\pa{\Gamma_{\infty,1}}$ is also a Pontryagin ring: there are natural maps of groups
$\Gamma_{g,1}\times\Gamma_{g',1}\to\Gamma_{g+g',1}$ inducing a Pontryagin product on the homology 
$H_*\pa{\Gamma_{\infty,1}}=\colim_{g\to\infty} H_*\pa{\gg}$.

Equation \ref{eq:doubleiso} gives
an isomorphism of modules over the ring
$H_*\pa{\Gamma_{\infty,1}}\otimes H_*(C_{\bullet}(\D))$. The action of the ring $H_*\pa{\Gamma_{\infty,1}}$
on $H_*\pa{\Gamma_{\infty,1}^{\bullet}}$ is induced on the colimit by 
the maps of groups $\gg\times\Gamma_{g',1}^m\to\Gamma_{g+g',1}^m$.
The corresponding action of $H_*\pa{\Gamma_{\infty,1}}$ on $H_*(C_{\bullet}(\D))\otimes H_*\pa{\Gamma_{\infty,1};\Sym_{\bullet}(\H)}$
comes from the action on $ H_*\pa{\Gamma_{\infty,1};\Sym_{\bullet}(\H)}$ which is induced on the colimit by the natural maps
\[
 H_*(\gg)\otimes H_*(\Gamma_{g',1};\Sym_m(H_1(\Sigma_{g',1})))\to H_*(\Gamma_{g+g',1};\Sym_m(H_1(\Sigma_{g+g',1}))).
\]

The right-hand side $H_*\pa{\Gamma_{\infty,1}}\otimes \pa{\bigoplus_{m\geq 0} H_*\pa{\mathfrak{S}_m;\Z_2[x_1,\dots,x_m]}}$
in Equation \ref{eq:doubleiso} is a free module over the ring $H_*\pa{\Gamma_{\infty,1}}\otimes H_*(C_{\bullet}(\D))$.
As a consequence we have that $H_*\pa{\Gamma_{\infty,1};\Sym_{\bullet}(\H_{\infty})}$ is a free module
over $H_*\pa{\Gamma_{\infty,1}}$; this last statement is of independent interest, and regards only the homology
of the group $\Gamma_{\infty,1}$ with twisted coefficients in $\Sym_{\bullet}(\H_{\infty})$.


We turn now to Theorem \ref{thm:BoTtwo}, and for the moment we assume again that $\F$ is a generic field and that the genus $g$
is fixed and finite. Considering
all values of $m$ at the same time, we can equivalently write
\begin{equation}
\label{eq:epsilonsplitting}
 H_*\pa{C_{\bullet}(\cF_{\S,\D});\F}\simeq \F[\epsilon]\otimes_{\F}\pa{\bigoplus_{m\geq 0} H_*\pa{C_m(\cF_{\S,\D}),C_{m-1}(\cF_{\S,\D})};\F}.
 \end{equation}

Here we use the convention $C_{-1}(\cF_{\S,\D})=\emptyset$, and we regard $C_{m-1}(\cF_{\S,\D})$ as a subspace of $C_m(\cF_{\S,\D})$ by using the map
$\mu^{\cF}$ in the statement of Theorem \ref{thm:BoTtwo}, with first input any fixed point $*\in C_1(\D)$. The class
$\epsilon$ is the canonical generator of $H_0(C_1(\D);\F)$.

In the case $\F=\Z_2$, Theorem \ref{thm:mainreformulation} improves Equation \ref{eq:epsilonsplitting} by exhibiting
$ H_*\pa{C_{\bullet}(\cF_{\S,\D})}$ as a free module over the ring $H_*(C_{\bullet}(\D))\cong\Z_2[Q^j\epsilon|j\geq 0]$,
rather than over its subring $\Z_2[\epsilon]$.

\section{A rational counterexample}
\label{sec:rational}
In this section we prove that a statement as in Theorem \ref{thm:Hbms*as*ggrep}
cannot hold if we consider homology with coefficients in $\Q$. We do not know
if the analogue of Theorem \ref{thm:main} holds in homology with coefficients in $\Q$
or in fields of odd characteristic.

We point out that $H_*(C_m(\Sigma_{g,1});\Q)$ has been computed \emph{as a bigraded $\Q$-vector space}
by B\"odigheimer, Cohen and Milgram in \cite{BCM}, and more recently by Knudson in \cite{Knudson}. A description of
these homology groups as a
$\gg$-representation seems to be still missing in the literature.

We will prove the following theorem:
\begin{thm}
 \label{thm:counterexample}
 Let $g\geq 2$ and $m\geq 2$; then $H_2(C_2(\Sigma_{g,1});\Q)$ is not a symplectic
 representation of $\gg$.
\end{thm}
\begin{proof}
 Let again $\S=\Sigma_{g,1}$. We will use the following strategy:
 \begin{itemize}
  \item we define a homology class $[a]\in H_2(C_2(\S);\Q)$ represented by a cycle $a$;
  \item we prove that $[a]\neq 0$ by computing the algebraic intersection of
  $[a]$ with a homology class
  in $\tH_2(\cms^{\infty};\Q)$: note that the manifold
  $\cms$ is orientable, hence Poincaré-Lefschetz duality holds
  also with rational coefficients;
  \item we define another homology class $[b]\in H_2(\cms;\Q)$ and show that
  $[b]$ is mapped to $[b]+2[a]$ by some element in the Torelli group $\mathcal{I}_{g,1}$.
 \end{itemize}
Recall that the Torelli group $\mathcal{I}_{g,1}$ is the kernel of the surjective homomorphism
$\gg\to Sp_{2g}(\Z)$ induced by the action of $\gg$ on $H_1\pa{\sg;\Z}$; if an element of the Torelli group acts non-trivially on
some class in $H_2(C_2(\S);\Q)$, the latter cannot be a symplectic representation of $\gg$.

We consider again $\T(\S)$ as model for $\mrS$. Let $c$ be an simple closed curve
representing the homology class $\u_1$, and assume that $c$ intersects $\U_1$ once
transversely and is disjoint from all other $\U_i$'s and from all $\V_i$'s. Let $c'$
be a parallel copy of $c$.

We consider the torus $a=c\times c'$ of configurations in $C_2(\S)$ having one point lying on $c$ and
one lying on $c'$; we let $[a]\in H_2(C_2(\S);\Q)$ be the fundamental class of $a$ associated with one
orientation of $a$.

Let $\tup=(0,\uu,\uv)$ with $\uu=(2,0)$ and $\uv=(0,0)$.
Then the map
\[
\Phi_{\tup}\colon (\Delta^{\tup},\partial\Delta^{\tup})\to (C_2(\S)^{\infty},\infty)
\]
maps the fundamental class in $H_2(\Delta^{\tup},\partial\Delta^{\tup};\Q)$ to a
class in $\tH_2(C_2(\S)^{\infty};\Q)$, that we call $[\kappa(0,\uu,\uv)]_{\Q}$.

The cell $e^{\tup}$ intersects once, transversely the torus $a$, therefore the
algebraic intersection $[a]\cap[\kappa(0,\uu,\uv)]_{\Q}$ is $\pm 1$,
where the signs depends on how we have chosen orientations on $a$, $e^{\tup}$ and $C_2(\S)$ itself.
In particular $[a]\neq 0$.

The action of $\gg$ on isotopy classes of non-separating simple closed curves is transitive, so we can repeat
the construction of the torus $a$ with any other copy of parallel, non-separating simple closed curves
$c$ and $c'$, and the resulting class $[a]$ will always be non-trivial in $H_2(C_2(\S);\Q)$.

See Figure \ref{fig:rational} to visualize the following discussion.
Let $d$ and $d'$ be disjoint, non-separating simple closed curves such that
the following holds: if we cut $\S$ along $d$ and $d'$ we separate $\S$ into two pieces, one of
which is a subsurface $\tilde{\S}\simeq\Sigma_{1,2}$ of genus 1, with boundary
curves $d$ and $d'$. Here we use that the genus of $\S$ is at least 2.

Suppose now that $c$ is a non-separating simple closed curve in $\tilde{\S}$,
and let $c',c''$ be two parallel copies of $c$ in $\Sigma_{1,2}$, one on each side of a small tubular
neighborhood of $c$.
Then $d,d',c',c''$ are the boundary of a subsurface $\check{\S}\simeq\Sigma_{0,4}\subset\tilde{\S}$.
Orient all curves $d,d',c,c',c''$ in such a way that the following equalities hold:
\begin{itemize}
 \item $[d]=[d']\in H_1(\S;\Q)$, as witnessed by the subsurface $\tilde{\S}$;
 \item $[c]=[c']=[c'']\in H_1(\S;\Q)$;
 \item $[d]-[d']+[c']-[c'']=0\in H_1(\S\setminus c;\Q)$, as witnessed by the subsurface $\check{\S}$.
\end{itemize}

\begin{figure}[ht]\centering
 \includegraphics[scale=1.0]{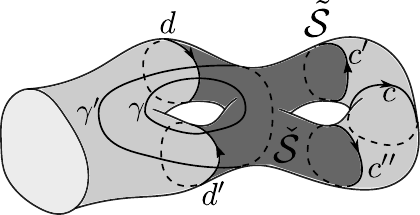}
 \caption{The curves $d,d',c,c',c'',\gamma,\gamma''$ and the subsurfaces $\tilde{\S},\check{\S}$.}
\label{fig:rational}
\end{figure}

The four tori $c\times c'$, $c\times c''$, $c\times d$ and $c\times d'$ are contained in $C_2(\S)$
and the equality
\[
 [c\times d]-[c\times d']+[c\times c']-[c\times c'']=0\in H_2(C_2(\S))
\]
holds, as witnessed by the homology $c\times \check{\S}\subset C_2(\S)$.

Moreover the classes $[c\times c']$ and $-[c\times c'']$ are \emph{equal}:
indeed there is an isotopy of $\S$ mapping $c$ to $c'$ and $c''$ to $c$,
as oriented curves; the class $[c\times c'']$ is mapped to the class $[c'\times c]=-[c\times c']$.

Let again $[a]=[c\times c']$ and let $[b]=[c\times d]$: then the class $[b']=[c\times d']$
is equal to $[b]+2[a]\in H_2(C_2(\S);\Q)$.

Consider now an element of the Torelli group that fixes $c$ and maps $d$ to $d'$ preserving
the orientation, for example the bounding pair $D_{\gamma}\circ D_{\gamma'}^{-1}$, where
$\gamma$ and $\gamma'$ are represented in Figure \ref{fig:rational}:
then the class $b$ is mapped to $b'=b+2a\neq b$.
\end{proof}

The previous proof works verbatim if we replace $\Q$ by any field of odd characteristic. Moreover
all the arguments in the previous proof can be adapted to show that $H_*(F_2(\S);\mathbb{F})$ is
not a symplectic representation of $\gg$ (see Definition \ref{defn:cms}), where
$\mathbb{F}$ is \emph{any} field, including $\Z_2$.

Therefore, at least for orientable surfaces $\sg$ of genus $g\geq 2$,
being a symplectic representation of the mapping class group seems to be a peculiarity
of the homology of \emph{unordered} configuration space and of the characteristic 2.

\bibliography{Bibliography.bib}{}
\bibliographystyle{plain}

\end{document}